\documentclass[12pt]{article}
\usepackage{latexsym, amssymb, amsthm}
\usepackage{dsfont}

\DeclareMathAlphabet\gothic{U}{euf}{m}{n}

\makeatletter
\def\eqnarray{\stepcounter{equation}\let\@currentlabel=\theequation
\global\@eqnswtrue
\tabskip\@centering\let\\=\@eqncr
$$\halign to \displaywidth\bgroup\hfil\global\@eqcnt\z@
  $\displaystyle\tabskip\z@{##}$&\global\@eqcnt\@ne
  \hfil$\displaystyle{{}##{}}$\hfil
  &\global\@eqcnt\tw@ $\displaystyle{##}$\hfil
  \tabskip\@centering&\llap{##}\tabskip\z@\cr}

\def\endeqnarray{\@@eqncr\egroup
      \global\advance\c@equation\m@ne$$\global\@ignoretrue}

\def\@yeqncr{\@ifnextchar [{\@xeqncr}{\@xeqncr[5pt]}}
\makeatother

\begin{document}
\bibliographystyle{tom}

\newtheorem{lemma}{Lemma}[section]
\newtheorem{thm}[lemma]{Theorem}
\newtheorem{cor}[lemma]{Corollary}
\newtheorem{prop}[lemma]{Proposition}

\theoremstyle{definition}

\newtheorem{remark}[lemma]{Remark}
\newtheorem{exam}[lemma]{Example}
\newtheorem{definition}[lemma]{Definition}

\newcommand{\gota}{\gothic{a}}
\newcommand{\gotb}{\gothic{b}}
\newcommand{\gotc}{\gothic{c}}
\newcommand{\gote}{\gothic{e}}
\newcommand{\gotf}{\gothic{f}}
\newcommand{\gotg}{\gothic{g}}
\newcommand{\gothh}{\gothic{h}}
\newcommand{\gotk}{\gothic{k}}
\newcommand{\gotm}{\gothic{m}}
\newcommand{\gotn}{\gothic{n}}
\newcommand{\gotp}{\gothic{p}}
\newcommand{\gotq}{\gothic{q}}
\newcommand{\gotr}{\gothic{r}}
\newcommand{\gots}{\gothic{s}}
\newcommand{\gott}{\gothic{t}}
\newcommand{\gotu}{\gothic{u}}
\newcommand{\gotv}{\gothic{v}}
\newcommand{\gotw}{\gothic{w}}
\newcommand{\gotz}{\gothic{z}}
\newcommand{\gotA}{\gothic{A}}
\newcommand{\gotB}{\gothic{B}}
\newcommand{\gotG}{\gothic{G}}
\newcommand{\gotL}{\gothic{L}}
\newcommand{\gotS}{\gothic{S}}
\newcommand{\gotT}{\gothic{T}}

\newcounter{teller}
\renewcommand{\theteller}{(\alph{teller})}
\newenvironment{tabel}{\begin{list}%
{\rm  (\alph{teller})\hfill}{\usecounter{teller} \leftmargin=1.1cm
\labelwidth=1.1cm \labelsep=0cm \parsep=0cm}
                      }{\end{list}}

\newcounter{tellerr}
\renewcommand{\thetellerr}{(\roman{tellerr})}
\newenvironment{tabeleq}{\begin{list}%
{\rm  (\roman{tellerr})\hfill}{\usecounter{tellerr} \leftmargin=1.1cm
\labelwidth=1.1cm \labelsep=0cm \parsep=0cm}
                         }{\end{list}}

\newcounter{tellerrr}
\renewcommand{\thetellerrr}{(\Roman{tellerrr})}
\newenvironment{tabelR}{\begin{list}%
{\rm  (\Roman{tellerrr})\hfill}{\usecounter{tellerrr} \leftmargin=1.1cm
\labelwidth=1.1cm \labelsep=0cm \parsep=0cm}
                         }{\end{list}}

\newcounter{proofstep}
\newcommand{\nextstep}{\refstepcounter{proofstep}\vertspace \par 
          \noindent{\bf Step \theproofstep} \hspace{5pt}}
\newcommand{\firststep}{\setcounter{proofstep}{0}\nextstep}

\newcommand{\Ni}{\mathds{N}}
\newcommand{\Qi}{\mathds{Q}}
\newcommand{\Ri}{\mathds{R}}
\newcommand{\Ci}{\mathds{C}}
\newcommand{\Ti}{\mathds{T}}
\newcommand{\Zi}{\mathds{Z}}
\newcommand{\Fi}{\mathds{F}}

\renewcommand{\proofname}{{\bf Proof}}

\newcommand{\vertspace}{\vskip10.0pt plus 4.0pt minus 6.0pt}

\newcommand{\simh}{{\stackrel{{\rm cap}}{\sim}}}
\newcommand{\ad}{{\mathop{\rm ad}}}
\newcommand{\Ad}{{\mathop{\rm Ad}}}
\newcommand{\alg}{{\mathop{\rm alg}}}
\newcommand{\clalg}{{\mathop{\overline{\rm alg}}}}
\newcommand{\Aut}{\mathop{\rm Aut}}
\newcommand{\arccot}{\mathop{\rm arccot}}
\newcommand{\capp}{{\mathop{\rm cap}}}
\newcommand{\rcapp}{{\mathop{\rm rcap}}}
\newcommand{\diam}{\mathop{\rm diam}}
\newcommand{\divv}{\mathop{\rm div}}
\newcommand{\dom}{\mathop{\rm dom}}
\newcommand{\codim}{\mathop{\rm codim}}
\newcommand{\RRe}{\mathop{\rm Re}}
\newcommand{\IIm}{\mathop{\rm Im}}
\newcommand{\tr}{{\mathop{\rm Tr \,}}}
\newcommand{\Tr}{{\mathop{\rm Tr \,}}}
\newcommand{\Vol}{{\mathop{\rm Vol}}}
\newcommand{\card}{{\mathop{\rm card}}}
\newcommand{\rank}{\mathop{\rm rank}}
\newcommand{\supp}{\mathop{\rm supp}}
\newcommand{\sgn}{\mathop{\rm sgn}}
\newcommand{\essinf}{\mathop{\rm ess\,inf}}
\newcommand{\esssup}{\mathop{\rm ess\,sup}}
\newcommand{\Int}{\mathop{\rm Int}}
\newcommand{\lcm}{\mathop{\rm lcm}}
\newcommand{\loc}{{\rm loc}}
\newcommand{\HS}{{\rm HS}}
\newcommand{\Trn}{{\rm Tr}}
\newcommand{\n}{{\rm N}}
\newcommand{\WOT}{{\rm WOT}}

\newcommand{\at}{@}

\newcommand{\mod}{\mathop{\rm mod}}
\newcommand{\spann}{\mathop{\rm span}}
\newcommand{\one}{\mathds{1}}

\hyphenation{groups}
\hyphenation{unitary}

\newcommand{\tfrac}[2]{{\textstyle \frac{#1}{#2}}}

\newcommand{\ca}{{\cal A}}
\newcommand{\cb}{{\cal B}}
\newcommand{\cc}{{\cal C}}
\newcommand{\cd}{{\cal D}}
\newcommand{\ce}{{\cal E}}
\newcommand{\cf}{{\cal F}}
\newcommand{\ch}{{\cal H}}
\newcommand{\chs}{{\cal HS}}
\newcommand{\ci}{{\cal I}}
\newcommand{\ck}{{\cal K}}
\newcommand{\cl}{{\cal L}}
\newcommand{\cm}{{\cal M}}
\newcommand{\cn}{{\cal N}}
\newcommand{\co}{{\cal O}}
\newcommand{\cp}{{\cal P}}
\newcommand{\cs}{{\cal S}}
\newcommand{\ct}{{\cal T}}
\newcommand{\cx}{{\cal X}}
\newcommand{\cy}{{\cal Y}}
\newcommand{\cz}{{\cal Z}}

\thispagestyle{empty}

\vspace*{1cm}
\begin{center}
{\Large\bf The Dirichlet-to-Neumann operator on $C(\partial \Omega)$
} \\[10mm]

\large W. Arendt$^1$ and A.F.M. ter Elst$^2$

\end{center}

\vspace{5mm}

\begin{center}
{\bf Abstract}
\end{center}

\begin{list}{}{\leftmargin=1.8cm \rightmargin=1.8cm \listparindent=10mm 
   \parsep=0pt}
\item
Let $\Omega \subset \Ri^d$ be an open bounded set with Lipschitz boundary~$\Gamma$.
Let $D_V$ be the Dirichlet-to-Neumann operator with respect to a purely second-order
symmetric divergence form operator with real Lipschitz continuous coefficients
and a positive potential~$V$.
We show that the semigroup generated by $-D_V$ leaves $C(\Gamma)$
invariant and that the restriction of this semigroup to $C(\Gamma)$
is a $C_0$-semigroup.
We investigate positivity and spectral properties of this semigroup.
We also present results where $V$ is allowed to be negative.
Of independent interest is a new criterium for semigroups to have 
a continuous kernel.

\end{list}

\vspace{4cm}
\noindent
July 2017

\vspace{5mm}

\vspace{5mm}

\vspace{15mm}

\noindent
{\bf Home institutions:}    \\[3mm]
\begin{tabular}{@{}cl@{\hspace{10mm}}cl}
1. & Institute of Applied Analysis  & 
  2. & Department of Mathematics   \\
& University of Ulm   & 
  & University of Auckland   \\
& Helmholtzstr.\ 18 & 
  & Private bag 92019  \\
& 89081 Ulm & 
  & Auckland  \\
& Germany  & 
  & New Zealand  \\[8mm]
\end{tabular}

\newpage

\section{Introduction} \label{Sdtnc1}

Let $\Omega \subset \Ri^d$ be an open set with Lipschitz boundary~$\Gamma$.
The Dirichlet-to-Neumann operator $D_0$ is the self-adjoint operator
that is defined in $L_2(\Gamma)$ as follows.
Let $\varphi,\psi \in L_2(\Gamma)$.
Then $\varphi \in \dom(D_0)$ and 
$D_0 \varphi = \psi$ if and only if
there exists a $u \in H^1(\Omega)$ such that $\Delta u = 0$ weakly 
on $\Omega$, with $\Tr u = \varphi$ and the weak normal derivative exists
with $\partial_\nu u = \psi$.
It turns out that the semigroup $S$ generated by $-D_0$ is submarkovian.
Hence it extends consistently to a contraction semigroup $S^{(p)}$ 
on $L_p(\Gamma)$ for all $p \in [1,\infty]$ and it is a $C_0$-semigroup
if $p \in [1,\infty)$.
By elliptic regularity the semigroup $S$ leaves the Banach space $C(\Gamma)$
of continuous functions on $\Gamma$ invariant.
Hence it is a natural question whether the restriction of $S$ 
to $C(\Gamma)$ is a $C_0$-semigroup.
As a special case of Theorem~\ref{tdtnc215}, we prove the following theorem.

\begin{thm} \label{tdtnc101}
Let $S$ be the semigroup generated by the Dirichlet-to-Neumann 
operator on an open bounded set with Lipschitz boundary~$\Gamma$.
Then $S$ leaves $C(\Gamma)$ invariant and the restriction 
of $S$ to $C(\Gamma)$ is a $C_0$-semigroup.
\end{thm}

If $\Omega$ has a $C^\infty$-boundary, then Theorem~\ref{tdtnc101} has been proved 
by Escher~\cite{Esc1} and Engel~\cite{Eng}.

Although $S$ leaves $C(\Gamma)$ invariant and $S$ is submarkovian, these two
facts do not imply that the restriction $T$ of $S$ to $C(\Gamma)$ is a $C_0$-semigroup,
since $C(\Gamma)$ is not reflexive.
One needs in addition that the generator of the restriction $T$ is 
densely defined.
This is the major problem that we solve in this paper.

Actually we prove several extensions of Theorem~\ref{tdtnc101}.
The first extension is that we replace the Laplacian by a divergence form 
operator $A$ with real symmetric Lipschitz continuous coefficients.
The second extension is that we add a potential $V \in L_\infty(\Omega,\Ri)$
to the divergence form operator and consider cases where the 
potential is negative (but still assuming the Dirichlet problem has a unique solution).
This means that given $\varphi \in L_2(\Gamma)$ we now 
solve the Dirichlet problem 
\[
\left[ \begin{array}{ll}
   (A + V) u= 0 & \mbox{weakly on } \Omega , \\[5pt]
   \Tr u = \varphi,  \\[5pt]
       \end{array} \right.
\]
and define the Dirichlet-to-Neumann operator $D_V$ by $D_V \varphi = \partial_\nu u$
on a suitable domain.
Using form methods one obtains that $-D_V$ generates a $C_0$-semigroup $S$ on $L_2(\Gamma)$
(see \cite{AEKS}).
The main point in this paper is to prove that the part of $D_V$ in $C(\Gamma)$
is densely defined in $C(\Gamma)$.
We prove this for all $V \in L_\infty(\Omega,\Ri)$, without any sign condition 
on $V$ (except assuming that the Dirichlet problem has a unique solution).
This is difficult even for the Laplacian since the normal is merely 
a measurable function on $\Gamma$.
For a rich class of potentials we then show that the restriction of $S$ to $C(\Gamma)$ 
is a $C_0$-semigroup on $C(\Gamma)$.

Attention is given to the special case where the semigroup~$S$
is positive.
Then we deduce that the Dirichlet-to-Neumann operator is resolvent positive 
on $C(\Gamma)$.

Another main point in this paper is the characterisation of those
semigroups in $L_2(K)$ which have a continuous kernel, where $K$ is a 
compact metric space.
This is done in an abstract framework.
Moreover, we find criteria for the irreducibility of the semigroup on $C(K)$.
Irreducibility is an important property which implies in particular
that the first eigenfunction is strictly positive.
We apply these results to the Dirichlet-to-Neumann operator 
but also to elliptic operators with Robin boundary conditions on~$\Omega$ if
$\Omega$ is connected.
So far, for Robin boundary conditions, 
strict positivity of the first eigenfunction in $C(\overline\Omega)$
was not known.
There is another reason to consider the Robin operator.
Even though $\Omega$ is connected, the boundary $\Gamma$ need not be be 
connected (an example is an annulus).
Still we are able to prove irreducibility for the Dirichlet-to-Neumann
semigroup on $C(\Gamma)$ and this is done with the Robin semigroup
on $C(\overline \Omega)$.
We should mention that irreducibility on $L_2$-spaces is much easier 
to obtain than on $C(K)$ (see \cite{Ouh5} Corollary~2.11 for elliptic 
operators and \cite{ArM2} Theorem~4.2 for the Dirichlet-to-Neumann operator).
The difference can be seen by the consequences for the first eigenfunction.
The irreducibility on $L_2$ merely implies that the first eigenfunction
is positive almost everywhere, whilst irreducibility on $C(K)$ implies
pointwise positive.
It is remarkable that our proof of this strict positivity (which is 
a purely elliptic property) involves considering the parabolic problem.

The paper is organised as follows.
In Section~\ref{Sdtnc5} we study in an abstract setting when a semigroup $S$ 
on $L_2(K)$ has a continuous kernel, where $K$ is a compact metric space.
If $S$ is positive and has a self-adjoint generator, then we characterise
when the restriction of $S$ to $C(K)$ is irreducible.
In Section~\ref{Sdtnc2} we consider the semigroup $S^V$ generated by $-D_V$,
where $D_V$ is the Dirichlet-to-Neumann operator with respect to a 
symmetric divergence form operator with coefficients $a_{kl} \in L_\infty(\Omega,\Ri)$
and potential $V \in L_\infty(\Omega,\Ri)$.
We show that $S^V$ has a continuous kernel and that the 
resolvent of $D_V$ leaves $C(\Gamma)$ invariant.
In Section~\ref{Sdtnc3} we prove that the domain of the part of $D_V$ 
in $C(\Gamma)$ is dense in $C(\Gamma)$ if the coefficients $a_{kl}$ are 
Lipschitz continuous. 
In Section~\ref{Sdtnc4} we prove an extension of Theorem~\ref{tdtnc101}
if $a_{kl} \in W^{1,\infty}(\Gamma)$ and the potential $V$ is positive
or slightly negative.
In Section~\ref{Sdtnc6} we study the Robin semigroup with boundary 
condition $\partial_\nu u + \beta \, \Tr u = 0$ without any 
sign condition on $\beta \in L_\infty(\Gamma,\Ri)$ and with 
coefficients of the divergence form operator in $L_\infty(\Omega,\Ri)$.
In the last section we show that $S^V$ is irreducible if merely $\Omega$ is 
connected and a positivity condition is satisfied.
Again the coefficients $a_{kl}$ are allowed to be measurable.

Using Poisson kernel bounds for the semigroup $S^V$, it is proved in 
\cite{EO7} that the semigroup $T$ is a holomorphic $C_0$-semigroup
on $C(\Gamma)$ if $\Omega$ has a $C^{1+\kappa}$-boundary for some $\kappa > 0$
and the coefficients $a_{kl}$ are merely H\"older continuous.
Thus more boundary smoothness of $\Omega$ is required in~\cite{EO7}.

\section{Continuous kernel and irreducibility} \label{Sdtnc5}

In this section we consider a semigroup $S$ on the space $L_2(K,\mu)$,
where $K$ is compact and $\mu$ is a finite Borel measure.
Our first aim is to investigate when $S$ has a continuous 
kernel.
Subsequently we asume that $S$ is positive (in the lattice sense)
and self-adjoint.
We will find criteria which imply that the first eigenfunction
is continuous and strictly positive.
In the sequel of this paper these two results will be applied to 
both the Dirichlet-to-Neumann operator and an elliptic operator with 
Robin boundary condition.

In general, by a {\bf semigroup} on a Banach space $X$ we understand 
simply a map $S \colon (0,\infty) \to \cl(X)$ satisfying 
$S_{t+s} = S_t \, S_s$ for all $t,s \in (0,\infty)$, 
without any further continuity assumption.
If $S$ is a semigroup on $L_2(K,\mu)$
we say that $S$ has a {\bf continuous kernel} if
for all $t > 0$ there exists a continuous function $k_t \colon K \times K \to \Ci$
such that for all $u \in L_2(K)$ the function $S_t u$ is given by
\[
(S_t u)(x)
= \int_K k_t(x,y) \, u(y) \, dy
\]
for almost every $x \in K$.
In many concrete situations regularity properties of kernels 
have been investigated,
but so far no characterisation for continuity of the kernel seems to be known.
The following theorem is such a charcaterisation 
in terms of a natural property, 
Condition~\ref{tdtnc240-3} in Theorem~\ref{tdtnc240},
which is frequently easy to verify.
Note that the semigroup does not have to be continuous in this theorem.

\begin{thm} \label{tdtnc240}
Let $K$ be a compact metric space and $\mu$ a finite Borel measure 
on~$K$.
Let $S$ be a semigroup on $L_2(K,\mu)$.
Then the following are equivalent.
\begin{tabeleq}
\item \label{tdtnc240-2}
The operator $S_t$ has a continuous kernel for all $t > 0$.
\item \label{tdtnc240-3}
There exists a $p_0 \in [2,\infty)$ such that 
$S_t L_{p_0}(K) \subset C(K)$ and $S_t^* L_{p_0}(K) \subset C(K)$ for all $t > 0$.
\item \label{tdtnc240-1}
$S_t L_2(K) \subset C(K)$ and $S_t^* L_2(K) \subset C(K)$ for all $t > 0$.
\end{tabeleq}
\end{thm}
\begin{proof}
`\ref{tdtnc240-2}$\Rightarrow$\ref{tdtnc240-3}'. 
Trivial.

`\ref{tdtnc240-3}$\Rightarrow$\ref{tdtnc240-1}'. 
We may assume that $p_0 \in \Ni$.
Let $t > 0$.
Then $S^*_t$ is bounded from $L_{p_0}(K)$ into $L_\infty(K)$,
so by duality $S_t$ extends to a bounded operator from $L_1(K)$ into $L_{q_0}(K)$,
where $\frac{1}{q_0} = 1 - \frac{1}{p_0}$.
Also $S_t$ is bounded from $L_{p_0}(K)$ into $L_\infty(K)$.
So by interpolation, given $p \in [1,p_0]$, the operator
$S_t$ extends to a bounded operator from $L_p(K)$ into $L_q(K)$,
where $\frac{1}{p} - \frac{1}{q} = \frac{1}{p_0}$.
Starting with $p=1$ and using the semigroup property, iteration gives that 
for all $t > 0$ and $k \in \{ 1,\ldots,p_0 \} $ the operator $S_t$ extends to a 
bounded operator from $L_1(K)$ into $L_q(K)$,
where $\frac{1}{q} = 1 - \frac{k}{p_0}$.
Therefore Condition~\ref{tdtnc240-1} is valid.

`\ref{tdtnc240-1}$\Rightarrow$\ref{tdtnc240-2}'.
Let $t > 0$.
Then $S^*_t L_2(K) \subset C(K) \subset L_\infty(K)$, 
so by duality $S_t$ extends to a bounded operator from $L_1(K)$ into $L_2(K)$,
also denoted by $S_t$.
Then by the semigroup property $S_{2t} L_1(K) \subset L_\infty(K)$.
Hence by the Dunford--Pettis theorem, for all $t > 0$ there exists a bounded 
measurable function $\tilde k_t \colon K \times K \to \Ci$ such that 
\[
(S_t u, v)_{L_2(K)} 
= \int_{K \times K} (u \otimes \overline v) \, \tilde k_t
\]
for all $u,v \in L_2(K)$.
Hence if $u \in L_2(K)$, then 
\begin{equation}
(S_t u)(x) 
= \int_K \tilde k_t(x,y) \, u(y) \, dy
\label{etdtnc240;1}
\end{equation}
for almost every $x \in K$
and by duality
\[
(S_t^* u)(x) 
= \int_K \tilde k^*_t(x,y) \, u(y) \, dy
\]
for almost every $x \in K$, where
$\tilde k^*_t(x,y) = \overline{\tilde k_t(y,x)}$ for all $(x,y) \in K \times K$ and 
$t > 0$.
If $t > 0$, then the semigroup property gives that 
\begin{equation}
\tilde k_{2t}(x,y)
= \int_K \tilde k_t(x,z) \, \tilde k_t(z,y) \, dz
\label{etdtnc240;2}
\end{equation}
for almost every $(x,y) \in K \times K$.
In particular, for almost all $x \in K$ it follows that (\ref{etdtnc240;2}) 
is valid for almost every $y \in K$.

Fix $t > 0$.
Since $S_t L_2(K) \subset C(K)$ it follows from the 
Riesz representation theorem that for all 
$x \in K$ there exists a $k^t_x \in \cl_2(K)$ such that 
\[
(S_t u)(x) = (u, k_x^t)_{L_2(K)}
\]
for all $u \in L_2(K)$ and $\|k^t_x\|_2 \leq \|S_t\|_{2 \to \infty}$.
Similarly, for all 
$y \in K$ there exists a $k^{*t}_y \in \cl_2(K)$ such that 
\[
(S^*_t u)(y) = (u, k^{*t}_y)_{L_2(K)}
\]
for all $u \in L_2(K)$.
Then $\|k^{*t}_x\|_2 \leq \|S^*_t\|_{2 \to \infty}$.
Next we use (\ref{etdtnc240;1}).
Let $u \in L_2(K)$.
Then 
\begin{equation}
\int_K \tilde k_t(x,y) \, u(y) \, dy
= (S_t u)(x) 
= (u, k_x^t)_{L_2(K)}
\label{etdtnc240;5}
\end{equation}
for almost every $x \in K$.
Since $C(K)$ is separable and $C(K)$ is dense in $L_2(K)$,
also the space $L_2(K)$ is separable.
Then by continuity and density 
it follows that (\ref{etdtnc240;5}) is valid for all $u \in L_2(K)$
for almost every $x \in K$.
Therefore $\overline{k_x^t} = \tilde k_t(x, \,\cdot\,)$ almost everywhere
for almost every $x \in K$.
Similarly, $\overline{k^{*t}_y} = \tilde k^*_t(y, \,\cdot\,)$ almost everywhere
for almost every $y \in K$.
Hence $k^{*t}_y = \tilde k_t(\, \cdot\, ,y)$ almost everywhere
for almost every $y \in K$.

The semigroup property (\ref{etdtnc240;2}) and Fubini's theorem give
that for almost every $x \in K$ it follows that 
\[
\tilde k_{2t}(x,y)
= \int_K \tilde k_t(x,z) \, \tilde k_t(z,y) \, dz
\]
for almost every $y \in K$.
Hence for almost every $x \in K$ it follows that 
\[
\tilde k_{2t}(x,y)
= \int_K \overline{k_x^t}(z) \, k^{*t}_y(z) \, dz
= (k^{*t}_y, k_x^t)_{L_2(K)}
\]
for almost every $y \in K$.
Define $\hat k_{2t} \colon K \times K \to \Ci$
by 
\[
\hat k_{2t}(x,y) 
= (k^{*t}_y, k_x^t)_{L_2(K)}
.  \]
We proved that $\tilde k_{2t}(x, \,\cdot\,) = \hat k_{2t}(x,\,\cdot\,)$
almost everywhere for almost every $x \in K$.
Clearly $|\hat k_{2t}(x,y)| \leq \|S_t\|_{2 \to \infty} \, \|S^*_t\|_{2 \to \infty}$
for all $x,y \in K$.

Since $S_t u \in C(K)$ obviously 
$x \mapsto (S_t u)(x) = (u, k_x^t)_{L_2(K)}$ is continuous for all $u \in L_2(K)$.
Hence if $y \in K$, then the function $x \mapsto \hat k_{2t}(x,y)$ is continuous
from $K$ into $\Ci$.
Similarly, for all $x \in K$ the function 
$y \mapsto \hat k_{2t}(x,y)$ is continuous
from $K$ into $\Ci$.
In particular, $\hat k_{2t}$ is a Carath\'eodory function and therefore
measurable (see \cite{AliprantisBorder} Lemma~4.51).
Because $\tilde k_{2t}(x, \,\cdot\,) = \hat k_{2t}(x,\,\cdot\,)$
almost everywhere for almost every $x \in K$,
one deduces from Fubini's theorem that 
$\tilde k_{2t} = \hat k_{2t}$ almost everywhere.

Define $k_{4t} \colon K \times K \to \Ci$ by
\[
k_{4t}(x,y) = \int_K \hat k_{2t}(x,z) \, \hat k_{2t}(z,y) \, dz
.  \]
Then the semigroup poperty~(\ref{etdtnc240;2}) gives
\[
\tilde k_{4t}(x,y)
= \int_K \tilde k_{2t}(x,z) \, \tilde k_{2t}(z,y) \, dz
= \int_K \hat k_{2t}(x,z) \, \hat k_{2t}(z,y) \, dz
= k_{4t}(x,y)
\]
for almost every $(x,y) \in K \times K$.
So $\tilde k_{4t} = k_{4t}$ almost everywhere.

Finally, for all $z \in K$ the function 
$(x,y) \mapsto \hat k_{2t}(x,z) \, \hat k_{2t}(z,y)$ is 
continuous from $K \times K$ into $\Ci$
and bounded 
by $\|S_t\|_{2 \to \infty}^2 \, \|S^*_t\|_{2 \to \infty}^2$.
Moreover, the measure is finite.
Hence by the Lebesgue dominated convergence theorem 
one deduces that $k_{4t}$ is continuous.
Therefore $\tilde k_{4t}$ has a continuous representative.
\end{proof}

\begin{remark} \label{rdtnc240.5}
Theorem~\ref{tdtnc240} is also valid if $K$ is replaced by a locally compact 
metric space~$X$ and $C(K)$ is replaced by $C_{\rm b}(X)$.
We do not know whether the condition that $\mu$ is a finite Borel
measure can be relaxed to $\mu$ being a regular measure.
\end{remark}

In the situation of Theorem~\ref{tdtnc240} it follows immediately that $S_t$ leaves 
$C(K)$ invariant for all $t > 0$.
Since kernel operators are compact, it follows that $(S_t|_{C(K)})_{t > 0}$ 
is a semigroup of compact operators in $C(K)$.
It is not clear, however, whether it is a $C_0$-semigroup, even if 
$S$ is a $C_0$-semigroup on $L_2(K)$.

\medskip

A subspace $I$ of a (general) Banach lattice $E$ is called an {\bf ideal}
if
\[
\left[ \begin{array}{l}
u \in I \mbox{ implies } |u| \in I \mbox{ and}  \\[5pt]
u \in I, \; v \in E \mbox{ and } 0 \leq v \leq u \mbox{ implies } v \in I.
       \end{array} \right.
\]
A semigroup on $E$ is called {\bf irreducible} if the only 
invariant closed ideals are $ \{ 0 \} $ and~$E$.
If $(X,\Sigma,\mu)$ is a measure space, $p \in [1,\infty)$ and 
$I \subset L_p(X)$, then $I$ is a closed ideal if and only if 
there exists a measurable subset $Y \subset X$ such that 
$I = \{ f \in L_p(X) : f|_Y = 0 \mbox{ a.e.} \} $
(see \cite{Schae2} Section~III.1 Example~1).
A subspace $I$ of $C(K)$ is a closed ideal of $C(K)$
if and only if there exists a closed set $B \subset K$ 
such that $I = \{ f \in C(K) : f|_B = 0 \} $ 
(see \cite{Schae2} Section~III.1 Example~2).
We refer to \cite{Nag} for much more information on irreducible semigroups.
An operator $B \colon E \to E$ is called {\bf positive} if $B f \geq 0$ for all 
$f \in E$ with $f \geq 0$.
A semigroup $S$ on $E$ is called {\bf positive} if $S_t$ is positive
for all $t > 0$.

In this paper we need a number of known properties of positive and irreducible 
semigroups when $E = L_2(K)$, where $K$ is a compact metric space.
For convenience and future reference we collect them in the next lemma.

\begin{lemma} \label{ldtnc528}
Let $S$ be a $C_0$-semigroup on $L_2(K,\mu)$, where $K$ is a 
compact metric space and $\mu$ is a finite Borel measure on $K$.
Suppose the generator $-A$ of $S$ is self-adjoint and that $S_t$ 
has a bounded kernel for all $t > 0$. 
Then one has the following.
\begin{tabel} 
\item \label{ldtnc528-1}
For all $t > 0$ the operator $S_t$ is a Hilbert--Schmidt operator.
\item \label{ldtnc528-2}
The operator $A$ has compact resolvent and 
$\min \sigma(A)$ is an eigenvalue.
\item \label{ldtnc528-3}
If $S$ is positive, then there exists an eigenfunction $u_1$ 
with eigenvalue $\min \sigma(A)$ such that $u_1 \geq 0$ almost everywhere.
\item \label{ldtnc528-4}
If $S$ is positive and irreducible, then the eigenvalue $\min \sigma(A)$ is simple.
Moreover, there exists an eigenfunction $u_1$ 
with eigenvalue $\min \sigma(A)$ such that $u_1(x) > 0$ for almost every 
$x \in K$.
\end{tabel}
\end{lemma}
\begin{proof}
`\ref{ldtnc528-1}' and `\ref{ldtnc528-2}' are easy.

`\ref{ldtnc528-3}'. 
This follows from the Krein--Rutman theorem, see for example
\cite{BatkaiKR} Theorem~12.15.

`\ref{ldtnc528-4}'. 
See \cite{BatkaiKR} Proposition~14.12(c) and Example~14.11(a).
\end{proof}

We emphasise that the eigenfunction $u_1$ 
in Statement~\ref{ldtnc528-3} is in general 
not unique, even not up to a positive constant.
If moreover $S_t L_2(K) \subset C(K)$ for all $t > 0$, 
then $u_1$ is continuous.
If $S$ is irreducible (on $L_2(K)$), then $u_1(x) > 0$ for almost all 
$x \in K$ by Lemma~\ref{ldtnc528}\ref{ldtnc528-3}.
Of course this does not imply that $u_1(x) > 0$ for all $x \in K$.
We will relate this strict positivity with the 
irreducibility of the semigroup on $C(K)$.
The main point of the following proposition is that the very 
weak nondegeneracy condition~\ref{pdtnc522-2} implies that the 
first eigenfunction is strictly positive.

\begin{prop} \label{pdtnc522}
Let $K$ be a compact connected metric space and $\mu$ a finite Borel measure 
on~$K$ with $\supp \mu = K$.
Let $S$ be a positive $C_0$-semigroup on $L_2(K,\mu)$
with self-adjoint generator $-A$.
Suppose that $S_t L_2(K) \subset C(K)$ for all $t > 0$.
Define 
\[
S^c_t = S_t|_{C(K)} \colon C(K) \to C(K)
\]
for all $t > 0$.
Then the following are equivalent.
\begin{tabeleq}
\item \label{pdtnc522-1}
The semigroup $S^c = (S^c_t)_{t > 0}$ is irreducible.
\item \label{pdtnc522-2}
For all $x \in K$ there exist $t > 0$ and $f \in C(K)$ such that 
$(S^c_t f)(x) \neq 0$.
\item \label{pdtnc522-3}
There exits a $\delta > 0$ such that $u_1(x) \geq \delta$
for all $x \in K$, where 
$u_1 \in L_2(K)$ is an eigenfunction with eigenvalue $\min \sigma(A)$
such that $u_1 \geq 0$ almost everywhere.
\item \label{pdtnc522-4}
For all $f \in C(K)$ with $f \geq 0$ and $f \neq 0$ it follows that 
$(S_t f)(x) > 0$ for all $t > 0$ and $x \in K$.
\end{tabeleq}
\end{prop}
\begin{proof}
`\ref{pdtnc522-1}$\Rightarrow$\ref{pdtnc522-4}'.
This is a variation of a theorem of Majewski and Robinson \cite{MajewskiRobinson}.
Let $x \in K$.
It follows from irreducibility that there exists a $t_1 > 0$ such that 
$(S^c_{t_1} f)(x) > 0$
(see \cite{Nag} Section~C-III Definition~3.1).
Let $\delta \in (0,t_1)$.
We shall show that $(S^c_t f)(x) = 0$ for all $t \in (\delta,\infty)$.
Set $t_0 = t_1 - \delta$ and $g = S^c_\delta f$.
Then $(S^c_{t_0} g)(x) > 0$.
Since $S^c$ has a holomorphic extension to a sector with 
values in $\cl(C(K))$, it follows from the proof
of Theorem~C-III.3.2(b) in \cite{Nag} that $(S^c_t g)(x) > 0$
for all $t > 0$.

`\ref{pdtnc522-4}$\Rightarrow$\ref{pdtnc522-3}'.
This is trivial.

`\ref{pdtnc522-3}$\Rightarrow$\ref{pdtnc522-2}'.
Take $f = u_1$.

`\ref{pdtnc522-2}$\Rightarrow$\ref{pdtnc522-1}'.
By Theorem~\ref{tdtnc240} the operator $S_t$ has a 
continuous kernel $k_t$ for all $t > 0$.
Let $B \subset K$ be a closed set with $\emptyset \neq B \neq K$.
Define 
\[
I = \{ f \in C(K) : f|_B = 0 \} 
.  \]
Suppose that the closed ideal $I$ is invariant under~$S$.
Define $g \in C(K)$ by $g(x) = d(x,B)$.
Then $g \in I$.
Since $K$ is connected there exists an $x_0 \in \partial B$.
Let $t > 0$.
Because $S_t g \in I$, one deduces that 
\[
\int_K k_t(x_0,y) \, d(y,B) \, d\mu(y)
= (S_t g)(x_0)
= 0
.  \]
Hence $k_t(x_0,y) = 0$ for a.e.\ $y \in K \setminus B$.
Since $k_t$ is continuous and $\mu$ is strictly positive on open sets
it follows that $k_t(x_0,y) = 0$ for all $y \in K \setminus B$.
Because $x_0 \in \partial B$ one establishes that $k_t(x_0,x_0) = 0$.
The semigroup property and symmetry then imply that 
\[
0 
= k_t(x_0,x_0)
= \int_K k_{t/2}(x_0,y) \, k_{t/2}(y,x_0) \, d\mu(y)
= \int_K |k_{t/2}(x_0,y)|^2 \, d\mu(y)
.  \]
Hence $k_{t/2}(x_0,y) = 0$ for almost every $y \in K$.
Consequently $(S_{t/2} f)(x_0) = 0$ for all $f \in C(K)$.
This is for all $t > 0$, which is a contradiction.
\end{proof}

Condition~\ref{pdtnc522-2} is automatically satisfied  if the semigroup
$S^c$ is a $C_0$-semigroup, because then 
$\lim_{t \downarrow 0} S^c_t \one = \one$ in $C(K)$.
As a consequence the semigroup is irreducible and $u_1(x) > 0$ for all 
$x \in K$.
This is surprising, since only the connectedness of $K$ is 
responsible for this property.
We state this as a corollary.

\begin{cor} \label{cdtnc522.5}
Let $K$ be a compact connected metric space and $\mu$ a finite Borel measure 
on~$K$ with $\supp \mu = K$.
Let $S$ be a positive $C_0$-semigroup on $L_2(K,\mu)$ with self-adjoint generator.
Suppose that $S_t L_2(K) \subset C(K)$ for all $t > 0$.
Define 
\[
S^c_t = S_t|_{C(K)} \colon C(K) \to C(K)
\]
for all $t > 0$.
If $S^c$ is a $C_0$-semigroup, then it is irreducible and 
$\min_{x \in K} u_1(x) > 0$.
\end{cor}

There is a remarkable consequence of irreducibility: the semigroup
$S$ extends to a consistent $C_0$-semigroup on $L_p(K)$ for all $p \in [1,\infty)$.

\begin{prop} \label{pdtnc524}
Let $K$ be a compact connected metric space and $\mu$ a finite Borel measure 
on~$K$ with $\supp \mu = K$.
Let $S$ be a positive $C_0$-semigroup on $L_2(K,\mu)$ with self-adjoint generator~$-A$.
Suppose that $S_t L_2(K) \subset C(K)$ for all $t > 0$ and that 
$S^c$ is irreducible.
Then for all $p \in [1,\infty)$ there exists a $C_0$-semigroup $S^{(p)}$ on $L_p(K)$
which is consistent to $S$.
Moreover, there exists an $M \geq 1$ such that 
$\|S^{(p)}_t\|_{p \to p} \leq M \, e^{-\lambda_1 t}$
for all $t > 0$ and $p \in [1,\infty)$, where $\lambda_1 = \min \sigma(A)$.
\end{prop}
\begin{proof}
Let $\delta > 0$ be as in Proposition~\ref{pdtnc522}\ref{pdtnc522-3}.
Let $0 \leq f \in L_\infty(K)$.
Then 
\[
0 \leq f
\leq \frac{\|f\|_\infty}{\delta} \, \delta
\leq \frac{\|f\|_\infty}{\delta} \, u_1
.  \]
Hence 
$S_t f \leq \frac{\|f\|_\infty}{\delta} \, S_t u_1
\leq \frac{\|f\|_\infty}{\delta} \, e^{-\lambda_1 t} \, u_1$
and $\|S_t f\|_\infty \leq M \, e^{-\lambda_1 t} \, \|f\|_\infty$
for all $t > 0$, where $M = \delta^{-1} \, \|u_1\|_\infty$.
Since $S$ is a self-adjoint semigroup, it follows by duality that 
$\|S_t f\|_1 \leq M \, e^{-\lambda_1 t} \, \|f\|_1$ for all $f \in L_2(K)$.
Then by interpolation
$\|S_t f\|_p \leq M \, e^{-\lambda_1 t} \, \|f\|_p$ for all $f \in L_2(K) \cap L_p(K)$.
Since the measure is finite the semigroup is a $C_0$-semigroup,
see \cite{Voi}.
\end{proof}

We emphasise that we do not assume in Proposition~\ref{pdtnc524}
that $S^c$ is a $C_0$-semigroup on $C(K)$.

\section{The Dirichlet-to-Neumann semigroup: invariance of $C(\Gamma)$} \label{Sdtnc2}

In this section we introduce the main setting of this paper and 
recall some known results for the Dirichlet-to-Neumann operator 
and the associated semigroup.

Let $\Omega \subset \Ri^d$ be a bounded open set with Lipschitz boundary.
For all $k,l \in \{ 1,\ldots,d \} $ let $a_{kl} \in L_\infty(\Omega,\Ri)$.
Suppose that 
\begin{equation}
a_{kl} = a_{lk}
\label{eSdtnc2;3}
\end{equation}
for all $k,l \in \{ 1,\ldots,d \} $
and that there exists a $\mu > 0$ such that 
\begin{equation}
\RRe \sum_{k,l=1}^d a_{kl}(x) \, \xi_k \, \overline{\xi_l} 
\geq \mu \, |\xi|^2
\label{eSdtnc2;1}
\end{equation}
for all $\xi \in \Ci^d$ and $x \in \Omega$.
Let $V \in L_\infty(\Omega,\Ri)$.
Define the forms $\gota,\gota_V \colon H^1(\Omega) \times H^1(\Omega) \to \Ci$ by 
\[
\gota(u,v) 
= \sum_{k,l=1}^d \int_\Omega a_{kl} \, (\partial_k u) \, \overline{\partial_l v}
\quad \mbox{and} \quad
\gota_V(u,v) 
= \gota(u,v) + \int_\Omega V \, u \, \overline v
.  \]
Let $A^N$ be the operator in $L_2(\Omega)$ associated with the form $\gota$
and let $A^D$ be the operator in $L_2(\Omega)$ associated with the form 
$\gota|_{H^1_0(\Omega) \times H^1_0(\Omega)}$.
Then $A^N + V$ is the operator associated with $\gota_V$ and 
$A^D + V$ is the operator associated with the form 
$\gota_V|_{H^1_0(\Omega) \times H^1_0(\Omega)}$.
We assume throughout this paper that 
\begin{equation}
0 \not\in \sigma(A^D + V)
.  
\label{eSdtnc2;2}
\end{equation}
Let $\Gamma$ be the boundary of $\Omega$.
We provide $\Gamma$ with the $(d-1)$-dimensional Hausdorff measure.
Let $D_V$ be the {\bf Dirichlet-to-Neumann operator} in $L_2(\Gamma)$ associated 
with $(\gota_V,\Tr)$.
This means the following.
If $\varphi,\psi \in L_2(\Gamma)$, then $\varphi \in D_V$ and $D_V \varphi = \psi$
if and only if there exists a $u \in H^1(\Omega)$ such that 
\[
\gota_V(u,v) 
= (\psi, \Tr v)_{L_2(\Gamma)}
\]
for all $v \in H^1(\Omega)$.
It follows from \cite{AEKS} Theorem~4.5, or 
\cite{BeE1} Theorem~5.10, that $D_V$ is a self-adjoint graph, 
which is indeed a self-adjoint operator because of the condition~(\ref{eSdtnc2;2}).
Moreover, $D_V$ is lower bounded by \cite{AEKS} Theorem~4.15.

We can give another description of the operator $D_V$, for which we need
the notion of a weak conormal derivative.
Let $H^{-1}(\Omega)$ be the dual space of $H^1_0(\Omega)$.
We define the operators $\ca, \ca+V \colon H^1(\Omega) \to H^{-1}(\Omega)$ by
\[
\langle \ca u,v \rangle
= \gota(u,v)
\quad \mbox{and} \quad 
\langle (\ca + V) u,v \rangle
= \gota_V(u,v)
.  \]
Let $u \in H^1(\Omega)$ and suppose that $\ca u \in L_2(\Omega)$.
Then we say that $u$ has a {\bf weak conormal derivative} 
if there exists a $\psi \in L_2(\Gamma)$ such that 
\[
\gota(u,v) - \int_\Omega (\ca u) \, \overline v
= \int_\Gamma \psi \, \overline{\tr v}
\]
for all $v \in H^1(\Omega)$.
By the Stone--Weierstrass it follows that the function $\psi$ is 
unique and we write $\partial_\nu u = \psi$.
Note that the conormal derivative depends on the coefficients $a_{kl}$,
which is suppressed in the notation.

With this notation the operator $A^N$ can be seen as the realization of $\ca$
in $L_2(\Omega)$ with Neumann boundary conditions, since
\[
\dom(A^N)
= \{ u \in H^1(\Omega) : \ca u \in L_2(\Omega) \mbox{ and } \partial_\nu u = 0 \}
  \]
and $A^N u = \ca u$ for all $u \in \dom(A^N)$.

The alluded characterisation of $D_V$ is as follows.

\begin{lemma} \label{ldtnc230}
Let $\varphi,\psi \in L_2(\Gamma)$.
Then the following are equivalent.
\begin{tabeleq}
\item \label{ldtnc230-1}
$\varphi \in \dom(D_V)$ and $D_V \varphi = \psi$.
\item \label{ldtnc230-2}
There exists a $u \in H^1(\Omega)$ such that $(\ca + V) u = 0$, 
$\Tr u = \varphi$ and $\partial_\nu u = \psi$.
\end{tabeleq}
\end{lemma}

We leave the easy proof to the reader.

\medskip

Let $S^V$ be the semigroup generated by $-D_V$.
In the next proposition we use elliptic regularity to show that 
the resolvent of $D_V$ leaves $C(\Gamma)$ invariant.

\begin{lemma} \label{ldtnc231}
For all $k,l \in \{ 1,\ldots,d \} $ let $a_{kl} \in L_\infty(\Omega,\Ri)$.
Let $V \in L_\infty(\Omega,\Ri)$.
Suppose {\rm (\ref{eSdtnc2;3})}, {\rm (\ref{eSdtnc2;1})} and {\rm (\ref{eSdtnc2;2})}
are valid.
Let $\omega \in \Ri$ be such that $\|S^V_t\|_{2 \to 2} \leq e^{\omega t}$ for all 
$t > 0$.
Let $\lambda \in (\omega,\infty)$ and $\psi \in C(\Gamma)$.
Then $(\lambda \, I + D_V)^{-1} \psi \in C(\Gamma)$.
\end{lemma}
\begin{proof}
Write $\varphi = (\lambda \, I + D_V)^{-1} \psi \in L_2(\Gamma)$.
Then $D_V \varphi = \psi - \lambda \, \varphi$.
There exists a unique $u \in H^1(\Omega)$ such that 
$\Tr u = \varphi$ and $\gota_V(u,v) = \int_\Gamma (\psi - \lambda \, \varphi) \, \overline{\Tr v}$
for all $v \in H^1(\Omega)$.
Then
\[
\gota(u,v) 
   + \int_\Omega V \, u \, \overline v 
   + \lambda \int_\Gamma \Tr u \, \overline{\Tr v}
= \int_\Gamma \psi \, \overline{\Tr v}
\]
for all $v \in H^1(\Omega)$.
Hence by \cite{Nit4} Theorem~3.14(ii) one deduces that $u \in C(\overline \Omega)$.
So $\varphi \in C(\Gamma)$.
\end{proof}

Also the semigroup $S^V$ leaves $C(\Gamma)$ invariant.
Even stronger, it maps $L_1(\Gamma)$ into $C(\Gamma)$.

\begin{prop} \label{pdtnc232}
For all $k,l \in \{ 1,\ldots,d \} $ let $a_{kl} \in L_\infty(\Omega,\Ri)$.
Let $V \in L_\infty(\Omega,\Ri)$.
Suppose {\rm (\ref{eSdtnc2;3})}, {\rm (\ref{eSdtnc2;1})} and {\rm (\ref{eSdtnc2;2})}
are valid.
Then $S^V_t L_2(\Gamma) \subset C(\Gamma)$ for all $t > 0$.
\end{prop}

For the proof we need the following lemma.

\begin{lemma} \label{ldtnc233}
Adopt the notation and assumptions of Proposition~\ref{pdtnc232}.
Suppose $d \geq 3$.
Let $q \in [\frac{2d}{d+2}, \frac{d}{2})$ and $\varepsilon > 0$.
Let $u \in H^1(\Omega)$ and 
$\psi \in L_2(\Gamma) \cap L_{\frac{(d-1) q}{d-q} + \varepsilon}(\Gamma)$.
Suppose that 
\[
\gota_V(u,v) = \int_\Gamma \psi \, \overline{\Tr v}
\]
for all $v \in H^1(\Omega)$.
Then $\Tr u \in L_{\frac{(d-1) q}{d-2q}}(\Gamma)$.
\end{lemma}
\begin{proof}
This is a special case of \cite{Nit4} Lemma~3.11.
\end{proof}

\begin{proof}[{\bf Proof of Proposition~\ref{pdtnc232}.}]
First we show that for all $t > 0$
and $\varphi \in L_2(\Gamma)$ 
there exists an $\varepsilon > 0$ such that 
$S^V_t \varphi \in L_{d-1+\varepsilon}(\Gamma)$.
For this we may assume that $d \geq 3$, since the case $d=2$ is trivial.
For all $n \in \{ 1,\ldots,d-1 \} $ define 
\[
q_n = \frac{2d}{d+3-n}
.  \]
Then $q_1 = \frac{2d}{d+2}$, $q_{d-2} = \frac{2d}{5}$ and $q_{d-1} = \frac{d}{2}$.
Moreover, $q_{n+1} = \frac{q_n d}{d - \frac{1}{2} \, q_n}$ for all $n \in \{ 1,\ldots,d-2 \} $.
We shall show that for all $t > 0$, $\varphi \in L_2(\Gamma)$ and $n \in \{ 1,\ldots,d-1 \} $
there exists an $\varepsilon > 0$ such that 
$S^V_t \varphi \in L_{\frac{(d-1) q_n}{d-q_n} + \varepsilon}(\Gamma)$.
The proof is by induction on $n$.

Since $\frac{(d-1) q_1}{d - q_1} = 2 \frac{d-1}{d} < 2$, the case $n = 1$
is trivial.
Let $n \in \{ 1,\ldots,d-2 \} $ and suppose that 
for all $t > 0$ and  $\varphi \in L_2(\Gamma)$ 
there exists an $\varepsilon > 0$ such that 
$S^V_t \varphi \in L_{\frac{(d-1) q_n}{d-q_n} + \varepsilon}(\Gamma)$.
Let $t > 0$ and $\varphi \in L_2(\Gamma)$.
Set $\psi = S^V_t \, D_V \, S^V_t \varphi$.
Then there exists an $\varepsilon > 0$ such that 
$\psi \in L_{\frac{(d-1) q_n}{d-q_n} + \varepsilon}(\Gamma)$ by the induction hypothesis.
Note that $D_V \, S^V_{2t} \varphi = \psi$.
So by definition there exists a $u \in H^1(\Omega)$ such that 
$\Tr u = S^V_{2t} \varphi$ and 
$\gota_V(u,v) = \int_\Gamma \psi \, \overline{\Tr v}$
for all $v \in H^1(\Omega)$.
Because $q_n \leq q_{d-2} < \frac{d}{2}$ one deduces from Lemma~\ref{ldtnc233}
that $\Tr u \in L_{\frac{(d-1) q_n}{d-2q_n}}(\Gamma)$.
Since 
$\frac{(d-1) q_{n+1}}{d - q_{n+1}} 
= \frac{(d-1) q_n}{d - q_n}
< \frac{(d-1) q_n}{d - 2 q_n}$,
there exists an $\varepsilon' > 0$ such that 
$S^V_{2t} \varphi = \Tr u \in L_{\frac{(d-1) q_{n+1}}{d-q_{n+1}} + \varepsilon'}(\Gamma)$,
which completes the induction step.
So by induction for all $t > 0$ and $\varphi \in L_2(\Gamma)$
there exists an $\varepsilon > 0$ such that 
$S^V_t \varphi \in L_{\frac{(d-1) q_{d-1}}{d-q_{d-1}} + \varepsilon}(\Gamma)$.
But $\frac{(d-1) q_{d-1}}{d-q_{d-1}} = d-1$.

Thus we proved for all $d \geq 2$, $t > 0$
and $\varphi \in L_2(\Gamma)$ that
there exists an $\varepsilon > 0$ such that 
$S^V_t \varphi \in L_{d-1+\varepsilon}(\Gamma)$.
Now one can argue once again as above and use this time \cite{Nit4} Lemma~3.10
to deduce that $S^V_{2t} \varphi \in C(\Gamma)$ for all $t > 0$
and $\varphi \in L_2(\Gamma)$.
\end{proof}

\begin{cor} \label{cdtnc330}
For all $k,l \in \{ 1,\ldots,d \} $ let $a_{kl} \in L_\infty(\Omega,\Ri)$.
Let $V \in L_\infty(\Omega,\Ri)$.
Suppose {\rm (\ref{eSdtnc2;3})}, {\rm (\ref{eSdtnc2;1})} and {\rm (\ref{eSdtnc2;2})}
are valid.
Then $S^V$ has a continuous kernel.
\end{cor}
\begin{proof}
This follows from Proposition~\ref{pdtnc232} and Theorem~\ref{tdtnc240}.
\end{proof}

For all $t > 0$ define $T^V_t \colon C(\Gamma) \to C(\Gamma)$
by 
\[
T^V_t = S^V_t|_{C(\Gamma)}
.  \]
Obviously $T^V = (T^V_t)_{t > 0}$ is a semigroup, but it is 
unclear whether it is a $C_0$-semigroup.
Define the part $D_{V,c}$ of $D_V$ in $C(\Gamma)$ by 
\[
\dom(D_{V,c})
= \{ \varphi \in C(\Gamma) \cap \dom(D_V) : D_V \varphi \in C(\Gamma) \}
\]
and $D_{V,c} \varphi = D_V \varphi$ for all $\varphi \in \dom(D_{V,c})$.
If $T^V$ is a $C_0$-semigroup, then $-D_{V,c}$ is the generator of $T^V$
and consequently $\dom(D_{V,c})$ is dense in $C(\Gamma)$.

\section{Density of the domain in $C(\Gamma)$} \label{Sdtnc3}

In this section we shall prove that the operator $D_{V,c}$ has dense domain
if the coefficients $a_{kl}$ are Lipschitz continuous.

\begin{thm} \label{tdtnc201}
For all $k,l \in \{ 1,\ldots,d \} $ let $a_{kl} \in W^{1,\infty}(\Omega,\Ri)$.
Let $V \in L_\infty(\Omega,\Ri)$.
Suppose {\rm (\ref{eSdtnc2;3})}, {\rm (\ref{eSdtnc2;1})} and {\rm (\ref{eSdtnc2;2})}
are valid.
Then the domain $\dom(D_{V,c})$ of the operator $D_{V,c}$ is dense in $C(\Gamma)$.
\end{thm}

For the proof we need a lot of preparation.
Throughout this section we adopt the assumptions of Theorem~\ref{tdtnc201}.

\medskip

We aim to prove that $D_{V,c}$ has a dense domain,
that is that there are sufficiently many $u \in H^1(\Omega)$ 
such that $(\ca + V) u = 0$, $\Tr u$ is continuous,
the function $u$ has a weak conormal derivative and $\partial_\nu u$ 
is continuous.
The next lemma gives existence of 
a class of functions on $\Omega$ with continuous trace, which 
have a weak conormal derivative and the conormal 
derivative is bounded (but not necessarily continuous).

\begin{lemma} \label{ldtnc203}
Let $u \in C^1(\overline \Omega) \cap H^2(\Omega)$.
Then $u$ has a weak conormal derivative and 
$\partial_\nu u \in L_\infty(\Gamma)$.
\end{lemma}
\begin{proof}
Since the $a_{kl} \in W^{1,\infty}(\Omega)$ it follows that 
$\ca u \in L_2(\Omega)$.
Let $v \in C^\infty(\overline \Omega)$.
For all $k \in \{ 1,\ldots,d \} $ define $F_k \colon \overline \Omega \to \Ci$ by 
$F_k = \overline v \, \sum_{l=1}^d a_{kl} \, (\partial_l u)$.
Then $F_k \in C(\overline \Omega) \cap H^1(\Omega)$.
Moreover, 
$\divv F 
= \sum_{k,l=1}^d a_{kl} \, (\partial_k u) \, \overline{\partial_l v} - (\ca u) \, \overline v
\in L_1(\Omega)$.
Hence the divergence theorem gives
\begin{eqnarray*}
\sum_{k,l=1}^d \int_\Omega a_{kl} \, (\partial_k u) \, \overline{\partial_l v}
   - \int_\Omega (\ca u) \, \overline v
& = & \int_\Omega \divv F   \nonumber  \\
& = & \sum_{k=1}^d \int_\Gamma \nu_k \, \Tr F_k
= \sum_{k,l=1}^d \int_\Gamma (\nu_k \, \Tr(a_{kl} \, \partial_l u)) \, \overline{\Tr v}
,
\end{eqnarray*}
where $\nu$ is the normal vector.
Then by density
\[
\sum_{k,l=1}^d \int_\Omega a_{kl} \, (\partial_k u) \, \overline{\partial_l v}
   - \int_\Omega (\ca u) \, \overline v
= \sum_{k,l=1}^d \int_\Gamma \Big( \nu_k \, \Tr(a_{kl} \, \partial_l u) \Big) \, \overline{\Tr v}
\]
for all $v \in H^1(\Omega)$.
So $u$ has a weak conormal derivative and 
$\partial_\nu u 
= \sum_{k,l=1}^d \nu_k \, \Tr(a_{kl} \, \partial_l u)
\in L_\infty(\Gamma)$.
\end{proof}

Our next aim is to show that one can approximate an element of $C(\Gamma)$
by functions $u|_\Gamma$, where $u \in C^1(\overline \Omega) \cap H^2(\Omega)$
and $(\ca + V) u = 0$.
We will show this in Lemma~\ref{ldtnc206}.
For such $u$ one deduces from the previous lemma that 
$u|_\Gamma \in \dom(D_V) \cap C(\Gamma)$ and 
$D_V (u|_\Gamma) = \partial_\nu u \in L_\infty(\Gamma)$.

The first ingredient is that the Lipschitz domain $\Omega$ can be approximated 
from outside by smooth domains.

\begin{lemma} \label{ldtnc205}
There exist $c_1,c_2 > 0$ and $\Omega_1,\Omega_2,\ldots \subset \Ri^d$
such that the following is valid.
\begin{tabel}
\item \label{ldtnc205-1}
For all $n \in \Ni$ the set $\Omega_n$ is open bounded with $C^\infty$-boundary.
Moreover, 
$\overline \Omega \subset \Omega_{n+1} \subset \Omega_n \subset \Omega + B(0,\frac{1}{n})$.
\item \label{ldtnc205-1.5}
$\bigcap_{n=1}^\infty \overline{\Omega_n} = \overline \Omega$.
\item \label{ldtnc205-2}
For all $n \in \Ni$ and $z \in \Gamma$ there exists a $z' \in \Omega_n^c$
such that $|z-z'| \leq \frac{c_1}{n}$.
\item \label{ldtnc205-3}
If $n \in \Ni$, $z \in \partial \Omega_n$ and $r \in (0,1]$,
then $|B(z,r) \setminus \Omega_n| \geq c_2 \, r^d$.
\end{tabel}
\end{lemma}
\begin{proof}
This is a straightforward consequence of \cite{Dok} Theorem~5,1.
\end{proof}

Since $\Omega$ has a Lipschitz boundary, one can extend the coefficients
$a_{kl}$ to bounded real valued Lipschitz continuous functions on $\Ri^d$,
which by abuse of notation we continue to denote by~$a_{kl}$.
Reducing $\mu$ if necessary, we may assume without loss of generality
that (\ref{eSdtnc2;1}) is valid for all $\xi \in \Ci^d$ and $x \in \Ri^d$.
Similarly we extend $V$ to a bounded real valued measurable function on $\Ri^d$, 
still denoted by~$V$.
If $\Omega' \subset \Ri^d$ is open, then we define similarly to $\gota$ 
the form $\gota_{\Omega'} \colon H^1(\Omega') \times H^1(\Omega') \to \Ci$
and define similarly the operators $A_{\Omega'}^D$ and $A_{\Omega'}^N$.
Moreover, define similarly the operator $\ca_{\Omega'} \colon H^1(\Omega') \to H^{-1}(\Omega')$.

If $\Omega',\Omega'' \subset \Ri^d$ are open with $\Omega' \subset \Omega''$,
then we will identify a self-adjoint operator in $L_2(\Omega')$ 
with a self-adjoint graph in $L_2(\Omega'')$ in a natural way,
see \cite{AEKS} Section~3, in particular Proposition~3.3. 
Moreover, we identify an element of $H^1_0(\Omega')$ with an element
in $H^1_0(\Omega'')$ by extending the function with zero.
Then $H^{-1}(\Omega'') \subset H^{-1}(\Omega')$.

\smallskip

If $\Omega_1,\Omega_2,\ldots \subset \Ri^d$ are as in Lemma~\ref{ldtnc205},
then the operators $A^D_{\Omega_n} + V$ in $L_2(\Omega_n)$ are 
a good approximation for the operator $A^D + V$ in $L_2(\Omega)$.
This is the content of the next two lemmas.
The first lemma is not new. 
We include the proof for completeness and refer to Daners \cite{Daners5}
for a systematic investigation of domain approximation.

\begin{lemma} \label{ldtnc211}
Let $\Omega_1,\Omega_2,\ldots \subset \Ri^d$ be open bounded sets
with $\overline \Omega \subset \Omega_{n+1} \subset \Omega_n$ for all 
$n \in \Ni$ and $\bigcap_{n=1}^\infty \overline{\Omega_n} = \overline \Omega$.
Let $\omega \in \Ri$ and suppose that $V + \omega \, \one_{\Omega_1} \geq \one_{\Omega_1}$.
Then
\[
\lim_{n \to \infty} (A^D_{\Omega_n} + V + \omega \, I)^{-1} 
= (A^D_\Omega + V + \omega \, I)^{-1} 
\]
in $\cl(L_2(\Omega_1))$.
\end{lemma}
\begin{proof}
Without loss of generality we may assume that $V \geq \one_{\Omega_1}$ and 
$\omega = 0$.
Let $f,f_1,f_2,\ldots \in L_2(\Omega_1)$ and suppose that
$\lim f_n = f$ weakly in $L_2(\Omega_1)$.
Let $n \in \Ni$.
Set $u_n = (A^D_{\Omega_n} + V)^{-1} f_n$.
Then $u_n \in H^1_0(\Omega_n) \subset H^1_0(\Omega_1)$.
Moreover, 
\begin{equation}
\gota_{\Omega_1}(u_n,v) + (V \, u_n, v)_{L_2(\Omega_1)} 
= (f_n, v)_{L_2(\Omega_1)}
\label{eldtnc211;1}
\end{equation}
for all $v \in H^1_0(\Omega_n)$.
Choose $v = u_n$.
Then 
\[
\mu \int_{\Omega_1} |\nabla u_n|^2 + \int_{\Omega_1} |u_n|^2
\leq \gota_{\Omega_1}(u_n) + (V \, u_n, u_n)_{L_2(\Omega_1)} 
= (f_n, u_n)_{L_2(\Omega_1)}
\leq \|f_n\|_{L_2(\Omega_1)} \, \|u_n\|_{L_2(\Omega_1)}
.  \]
Hence $\|u_n\|_{L_2(\Omega_1)} \leq \|f_n\|_{L_2(\Omega_1)}$
and $\mu \int_{\Omega_1} |\nabla u_n|^2 \leq \|f_n\|_{L_2(\Omega_1)}^2$.
Since $(f_n)_{n \in \Ni}$ is bounded in $L_2(\Omega_1)$, it follows 
that the sequence $(u_n)_{n \in \Ni}$ is bounded in $H^1_0(\Omega_1)$.
Passing to a subsequence, if necessary, we may assume that there exists a 
$u \in H^1_0(\Omega_1)$ such that 
$\lim u_n = u$ weakly in $H^1_0(\Omega_1)$.
Because $\Omega_1$ is bounded, one then obtains that $\lim u_n = u$ 
(strongly) in $L_2(\Omega_1)$.
Since $\supp u_n \subset \overline{\Omega_m}$ for all $n,m \in \Ni$ with 
$n \geq m$, it follows that $\supp u \subset \overline{\Omega_m}$ for all $m \in \Ni$.
So $\supp u \subset \bigcap_{m=1}^\infty \overline{\Omega_m} = \overline \Omega$.
Hence $u \in H^1_0(\Omega)$ since $\Omega$ has a Lipschitz boundary.
Let $v \in H^1_0(\Omega)$.
Then $v \in H^1_0(\Omega_n)$ for all $n \in \Ni$.
Use (\ref{eldtnc211;1}) and take the limit $n \to \infty$.
Then
\[
\gota(u,v) + (V \, u,v)_{L_2(\Omega)} = (f,v)_{L_2(\Omega)}
.  \]
So $u \in \dom(A^D + V)$ and $(A^D + V) u = f|_\Omega$.
Therefore $u = (A^D + V)^{-1} f$.

Choosing $f_n = f$ for all $n \in \Ni$ we proved that 
$\lim_{n \to \infty} (A^D_{\Omega_n} + V)^{-1} f = (A^D + V)^{-1} f$
in $L_2(\Omega_1)$ for all $f \in L_2(\Omega_1)$.

Finally, suppose that not 
$\lim_{n \to \infty} (A^D_{\Omega_n} + V)^{-1} 
= (A^D_\Omega + V)^{-1}$ 
in $\cl(L_2(\Omega_1))$.
There there are $\varepsilon > 0$ and $f_1,f_2,\ldots \in L_2(\Omega_1)$ 
such that $\|f_n\|_{L_2(\Omega_1)} = 1$ and 
\[
\|(A^D_{\Omega_n} + V)^{-1} f_n 
    - (A^D_\Omega + V)^{-1} f_n\|_{L_2(\Omega_1)} 
\geq \varepsilon
\]
for all $n \in \Ni$.
Passing to a subsequence, if necessary, we may assume that there exists 
an $f \in L_2(\Omega_1)$ such that 
$\lim f_n = f$ weakly in $L_2(\Omega_1)$.
Then $\lim_{n \to \infty} (A^D_{\Omega_n} + V)^{-1} f_n = (A^D + V)^{-1} f$
in $L_2(\Omega_1)$ by the above.
Since $(A^D + V)^{-1}$ is compact, also 
$\lim_{n \to \infty} (A^D_\Omega + V)^{-1} f_n = (A^D + V)^{-1} f$ in $L_2(\Omega_1)$.
So $\lim_{n \to \infty} \|(A^D_{\Omega_n} + V)^{-1} f_n 
    - (A^D_\Omega + V)^{-1} f_n\|_{L_2(\Omega_1)} = 0$.
This is a contradiction.
\end{proof}

\begin{lemma} \label{ldtnc212}
Let $\Omega_1,\Omega_2,\ldots \subset \Ri^d$ be open bounded sets
with $\overline \Omega \subset \Omega_{n+1} \subset \Omega_n$ for all 
$n \in \Ni$ and $\bigcap_{n=1}^\infty \overline{\Omega_n} = \overline \Omega$.
Then there exists a $\delta > 0$ such that 
\[
\sigma(A^D_{\Omega_n} + V) \cap (-\delta,\delta) = \emptyset
\]
for all large $n \in \Ni$.
\end{lemma}
\begin{proof}
For all $n \in \Ni$ the self-adjoint operators $A^D_{\Omega_n} + V$
and $A^D + V$ are lower bounded by $- \|V\|_{L_\infty(\Omega_1)}$
and have compact resolvent.
Hence they have a discrete spectrum.
Let $n \in \Ni$.
For all $m \in \Ni$ let $\lambda^{(n)}_m$ be the $m$-th eigenvalue 
of $A^D_{\Omega_n} + V$, counted with multiplicity.
Define similarly $\lambda_m$ with respect to $A^D + V$.
Since $\lim_{n \to \infty} (A^D_{\Omega_n} + V + \omega \, I)^{-1} 
= (A^D_\Omega + V + \omega \, I)^{-1}$ in $\cl(L_2(\Omega_1))$ with 
$\omega = \|V\|_{L_\infty(\Omega_1)} + 1$ by Lemma~\ref{ldtnc211}, it 
follows that $\lim_{n \to \infty} \lambda^{(n)}_m = \lambda_m$
for all $m \in \Ni$.
For a short proof of this well known fact see \cite{EM1}.

By assumption $0 \not\in \sigma(A^D + V)$.
Hence there exists a $\delta > 0$ such that 
$\sigma(A^D_\Omega + V) \cap (-\delta,\delta) = \emptyset$.
Since the eigenvalues converge, then also 
$\sigma(A^D_{\Omega_n} + V) \cap (-\delta,\delta) = \emptyset$ for 
all large $n \in \Ni$.
\end{proof}

The next lemma is a small extension of a special case of Theorem~1.2
in \cite{ERe2}.

\begin{lemma} \label{ldtnc209}
For all $c,d,\mu,M > 0$ and $p \in (\frac{d}{2} \vee 2,\infty)$ there exist $\alpha \in (0,1)$
and $c_1 > 0$ such that the following is valid.

Let $\Omega \subset \Ri^d$ be open non-empty and suppose that 
$|B(z,r) \setminus \Omega| \geq c \, r^d$ for all $z \in \partial \Omega$ and 
$r \in (0,1]$.
Let $V \in L_\infty(\Omega)$ with $\|V\|_{L_\infty(\Omega)} \leq M$.
For all $k,l \in \{ 1,\ldots,d \} $ let $a_{kl} \in L_\infty(\Omega,\Ri)$
with $\|a_{kl}\|_{L_\infty(\Omega)} \leq M$ and suppose that 
$\RRe \sum_{k,l=1}^d a_{kl}(x) \, \xi_k \, \overline{\xi_l} \geq \mu \, |\xi|^2$
for almost all $x \in \Omega$ and all $\xi \in \Ci^d$.
Let $f \in L_p(\Omega) \cap L_2(\Omega)$
and $u \in H^1_0(\Omega)$.
Suppose that 
\[
\sum_{k,l=1}^d \int_\Omega a_{kl} \, (\partial_k u) \, \overline{\partial_l v}
   + \int_\Omega V \, u \, \overline v
= \int_\Omega f \, \overline v
\]
for all $v \in H^1_0(\Omega)$.
Then $u \in C^\alpha(\Omega)$ and 
\[
|||u|||_{C^\alpha(\Omega)}
\leq c_1 \Big( \|u\|_{H^1(\Omega)} + \|f\|_{L_p(\Omega)} \Big)
,  \]
where 
\begin{equation}
|||u|||_{C^\alpha(\Omega)} 
= \sup \Big\{ \frac{|u(x) - u(y)|}{|x-y|^\alpha} : x,y \in \Omega 
   \mbox{ and } 0 < |x-y| \leq 1 \Big\} 
.  \
\label{eldtnc209;1}
\end{equation}
\end{lemma}
\begin{proof}
If $V = 0$, then this is a special case of \cite{ERe2} Theorem~1.2
with the choice $\Gamma = \emptyset$, $\Upsilon = \Omega$
and $\zeta = 2$.
If $V \neq 0$, then one has to replace $f$ by $f - V \, u$ and iterate,
using Proposition~3.2 of \cite{ERe2}.
\end{proof}

Now we are able to prove that one can
approximate elements in $C(\Gamma)$ by elements 
$\varphi \in \dom(D_V) \cap C(\Gamma)$ with $D_V \varphi \in L_\infty(\Gamma)$.

\begin{lemma} \label{ldtnc206}
Let $\varphi \in C(\Gamma)$ and $\varepsilon > 0$.
Then there exists a $u \in C^1(\overline \Omega) \cap H^2(\Omega)$
such that $(\ca + V) u = 0$
and $\|u|_\Gamma - \varphi\|_{C(\Gamma)} < \varepsilon$.
\end{lemma}
\begin{proof}
Since $ \{ F|_\Gamma : F \in C^2(\Ri^d) \} $ is dense in $C(\Gamma)$
by the Stone--Weierstra\ss\ theorem, we may assume that there 
exists an $F \in C^2(\Ri^d)$ such that $\varphi = F|_\Gamma$.

Let $c_1,c_2 > 0$ and $\Omega_1,\Omega_2,\ldots \subset \Ri^d$ be as in 
Lemma~\ref{ldtnc205}.
By Lemma~\ref{ldtnc212} there exists a $\delta > 0$ such that 
$\sigma(A^D_{\Omega_n} + V) \cap (-\delta,\delta) = \emptyset$ 
for all large $n \in \Ni$.
Without loss of generality we may assume that 
$\sigma(A^D_{\Omega_n} + V) \cap (-\delta,\delta) = \emptyset$ 
for all $n \in \Ni$.
Then in particular $A^D_{\Omega_n} + V$ is invertible for all $n \in \Ni$.
Let $n \in \Ni$.
Define $G_n \in L_2(\Omega_n)$ by 
\[
G_n = - \sum_{k,l=1}^d \partial_l \, a_{kl} \, \partial_k(F|_{\Omega_n})
.  \]
Since $F \in C^2(\Ri^d)$ and $a_{kl} \in W^{1,\infty}(\Ri^d)$ one 
indeed obtains that $G_n \in L_2(\Omega_n)$.
Even stronger, $G_n \in L_{d+1}(\Omega_n)$.
Since $A^D_{\Omega_n} + V$ is invertible, we can define
\[
w_n = (A^D_{\Omega_n} + V)^{-1} (G_n + V \, F)
.  \]
Then $w_n \in H^1_0(\Omega_n) \cap C_0(\Omega_n)$, where the continuity 
follows for example from Lemma~\ref{ldtnc209}.
Moreover, 
\[
\gota_{\Omega_n}(w_n,v) + \int_{\Omega_n} V \, w_n \, \overline v
= \int_{\Omega_n} (G_n + V \, F) \, \overline v
=  \gota_{\Omega_n}(F|_{\Omega_n},v) + \int_{\Omega_n} V \, F \, \overline v
\]
for all $v \in H^1_0(\Omega_n)$.
Let $u_n = F|_{\Omega_n} - w_n$.
Then $(\ca_{\Omega_n} + V) u_n = 0$.
So $\ca_{\Omega_n} u_n = - V \, u_n$ and hence 
$u_n \in W^{2,p}_\loc(\Omega_n)$ for all $p \in [1,\infty)$
by elliptic regularity.
In particular $u_n|_\Omega \in C^1(\overline \Omega) \cap H^2(\Omega)$.
Note that $u_n - F|_{\Omega_n} = - w_n$.
By Lemmas~\ref{ldtnc209} and \ref{ldtnc205}\ref{ldtnc205-3}
there exist $\alpha \in (0,1)$ and $c_3 > 0$,
{\em independent} of $n$, 
such that 
\begin{equation}
|||w_n|||_{C^\alpha(\Omega_n)}
\leq c_3 \Big( \|G_n + V \, F\|_{L_{d+1}(\Omega_n)}
             + \|w_n\|_{H^1(\Omega_n)} 
       \Big)
\label{eldtnc206;2}
\end{equation}
for all $n \in \Ni$, where $|||w_n|||_{C^\alpha(\Omega_n)}$ is defined
as in (\ref{eldtnc209;1}).
Clearly $\|G_n + V \, F\|_{L_{d+1}(\Omega_n)} \leq \|G_1 + V \, F\|_{L_{d+1}(\Omega_1)}$
for all $n \in \Ni$.
We next show that $(\|w_n\|_{H^1(\Omega_n)})_{n \in \Ni}$ is bounded.

Let $n \in \Ni$.
Since $\sigma(A^D_{\Omega_n} + V) \cap (-\delta,\delta) = \emptyset$
it follows that $\|(A^D_{\Omega_n} + V)^{-1}\| \leq \delta^{-1}$.
Therefore
\begin{equation}
\|w_n\|_{L_2(\Omega_n)} 
\leq \|(A^D_{\Omega_n} + V)^{-1}\| \, \|G_n + V \, F\|_{L_2(\Omega_n)} 
\leq \tfrac{1}{\delta} \, \|G_1 + V \, F\|_{L_2(\Omega_1)} 
.  
\label{eldtnc206;3}
\end{equation}
Set $\omega = \|V\|_{L_\infty(G_1)}$.
Then 
\begin{eqnarray*}
\mu \int_{\Omega_n} |\nabla w_n|^2
& \leq & \gota_{\Omega_n}(w_n)
\leq \gota_{\Omega_n}(w_n) + \int_{\Omega_n} (V + \omega \, \one_{\Omega_n}) \, |w_n|^2  \\
& = & \int_{\Omega_n} (G_n + V \, F + \omega \, w_n) \, \overline{w_n}  \\
& \leq & \Big( \|G_1 + V \, F\|_{L_2(\Omega_1)} + \omega \, \|w_n\|_{L_2(G_n)} \Big)
    \, \|w_n\|_{L_2(G_n)}
.  
\end{eqnarray*}
Together with (\ref{eldtnc206;3}) one concludes
the sequence $(\|w_n\|_{H^1(\Omega_n)})_{n \in \Ni}$ is bounded.

Using (\ref{eldtnc206;2}) there exists a $c_4 > 0$ such that 
$|||w_n|||_{C^\alpha(\Omega_n)} \leq c_4$ uniformly for all $n \in \Ni$.
Now let $z \in \Gamma$.
By Lemma~\ref{ldtnc205}\ref{ldtnc205-2} there exists a $z' \in \Omega_n^c$ such that 
$|z-z'| \leq \frac{c_1}{n}$.
Hence if $n \geq c_1$, then 
$|w_n(z)| 
= |w_n(z) - w_n(z')|
\leq |||w_n|||_{C^\alpha(\Omega_n)} \, |z-z'|^\alpha
\leq c_4 \, c_1^\alpha \, n^{-\alpha}$.
Therefore $\lim_{n \to \infty} \|w_n|_\Gamma\|_{C(\Gamma)} = 0$.
Hence $\lim_{n \to \infty} \|u_n|_\Gamma - F|_\Gamma\|_{C(\Gamma)} = 0$.
So choose $u = u_n|_{\overline \Omega}$ with $n$ large enough.
\end{proof}

We need one more lemma before we can prove density 
of $\dom(D_{V,c})$ in $C(\Gamma)$.
The main aim in the lemma is to solve the Neumann problem with respect to 
$A^N + V$ for functions $\psi \in L_p(\Gamma)$. 
If $p$ is large enough then solutions are continuous on $\overline \Omega$.
We choose $p = d$.
As expected, the kernel of of $A^N + V$ gives problems, so we 
take orthogonal complements.

\begin{lemma} \label{ldtnc208}
Define 
\[
H^1_{V\perp}(\Omega)
= \{ u \in H^1(\Omega) : (u,v)_{H^1(\Omega)} = 0 
     \mbox{ for all } v \in \ker(A^N + V) \}
\]
and 
\[
L_{d,V\perp}(\Gamma)
= \{ \tau \in L_d(\Gamma) : (\tau,\tr v)_{L_2(\Gamma)} = 0 
     \mbox{ for all } v \in \ker(A^N + V) \}
.  \]
Then one has the following.
\begin{tabel}
\item \label{ldtnc208-1}
$\ker(A^N + V) \subset C(\overline \Omega)$ is finite dimensional.
\item \label{ldtnc208-7}
If $u \in \ker(A^N + V)$, then $\tr u \in \dom(D_{V,c})$.
\item \label{ldtnc208-1.5}
If $\tau \in L_{d,V\perp}(\Gamma)$ and $\varepsilon > 0$, then there 
exists a $\tau' \in C(\Gamma) \cap L_{d,V\perp}(\Gamma)$ such that 
$\|\tau - \tau'\|_{L_d(\Gamma)} < \varepsilon$.
\item \label{ldtnc208-2}
For all $\tau \in L_{d,V\perp}(\Gamma)$ there exists a unique 
$u \in H^1_{V\perp}(\Omega)$ such that 
$\gota_V(u,v) = \int_\Gamma \tau \, \overline{\tr v}$
for all $v \in H^1(\Omega)$.
\end{tabel}
Define $E \colon L_{d,V\perp}(\Gamma) \to H^1_{V\perp}(\Omega)$ such 
that 
\[
\gota_V(E \tau, v)
= \int_\Gamma \tau \, \overline{\tr v}
\]
for all $v \in H^1(\Omega)$.
\begin{tabel}
\setcounter{teller}{4}
\item \label{ldtnc208-3}
The map $E$ is continuous.
\item \label{ldtnc208-4}
If $\tau \in L_{d,V\perp}(\Gamma)$, then $E \tau \in C(\overline \Omega)$.
\item \label{ldtnc208-5}
The map $E$ is continuous from $L_{d,V\perp}(\Gamma)$ into $C(\overline \Omega)$.
\item \label{ldtnc208-6}
If $\tau \in L_{d,V\perp}(\Gamma)$, then 
$\tr E \tau \in \dom(D_V)$ and $D_V \tr E \tau = \tau$.
\end{tabel}
\end{lemma}
\begin{proof}
`\ref{ldtnc208-1}'.
The operator $A^N + V$ has compact resolvent. 
Hence its kernel is finite dimensional. 
The inclusion follows from \cite{Nit4} Theorem~3.14(ii).

`\ref{ldtnc208-7}'.
If $v \in H^1(\Omega)$, then $\gota_V(u,v) = ((A^N + V) u, v)_{L_2(\Omega)} = 0$.
Therefore $\tr u \in \dom(D_V)$ and $D_V \tr u = 0$.
Since $\tr u \in C(\Gamma)$ by Statement~\ref{ldtnc208-1} and obviously 
the zero function is continuous, one deduces that $\tr u \in \dom(D_{V,c})$.

`\ref{ldtnc208-1.5}'.
By Statement~\ref{ldtnc208-1} there exist $N \in \Ni_0$ and 
$\varphi_1,\ldots,\varphi_N \in \tr \ker(A^N + V)$ such that 
$\varphi_1,\ldots,\varphi_N$ is a basis for $\tr \ker(A^N + V)$.
We may assume without loss of generality that 
$\varphi_1,\ldots,\varphi_N$ is orthonormal in $L_2(\Gamma)$.
Since $C(\Gamma)$ is dense in $L_d(\Gamma)$ there exists a $\tau'' \in C(\Gamma)$
such that $\|\tau - \tau''\|_{L_d(\Gamma)} < \varepsilon$.
For all $k \in \{ 1,\ldots,N \} $ set 
$c_k = (\tau'', \varphi_k)_{L_2(\Gamma)}$.
Then 
$|c_k| 
= |(\tau'' - \tau, \varphi_k)_{L_2(\Gamma)}|
\leq \varepsilon \, \|\varphi_k\|_{L_p(\Gamma)}$,
where $p \in (1,\infty)$ is the dual exponent of~$d$.
Set $\tau' = \tau'' - \sum_{k=1}^N c_k \, \varphi_k$.
Then 
\[
\|\tau - \tau'\|_{L_d(\Gamma)}
\leq \|\tau - \tau''\|_{L_d(\Gamma)} 
   + \sum_{k=1}^N |c_k| \, \|\varphi_k\|_{L_d(\Gamma)}
\leq \Big( 1 + \sum_{k=1}^N \|\varphi_k\|_{L_d(\Gamma)} \, \|\varphi_k\|_{L_p(\Gamma)} 
     \Big) \varepsilon
\]
and $\tau' \in C(\Gamma) \cap L_{d,V\perp}(\Gamma)$.

`\ref{ldtnc208-2}'.
Define the form $\gotb \colon H^1_{V\perp}(\Omega) \times H^1_{V\perp}(\Omega) \to \Ci$
by $\gotb = \gota_V|_{H^1_{V\perp}(\Omega) \times H^1_{V\perp}(\Omega)}$.
Then $\gotb$ is a continuous symmetric sesquilinear form.
Hence there exists a $T \in \cl(H^1_{V\perp}(\Omega))$ such that 
$\gotb(u,v) = (T u, v)_{H^1_{V\perp}(\Omega)}$ for all $u,v \in H^1_{V\perp}(\Omega)$.

We next show that $T$ is injective.
Indeed, if $u \in \ker T$, then
$\gota_V(u,v) = 0$ for all $v \in H^1_{V\perp}(\Omega)$.
Obviously $\gota_V(u,v) = (u, (A^N + V) v)_{L_2(\Omega)} = 0$ for all 
$v \in \ker(A^N + V)$.
Since $H^1(\Omega) = H^1_{V\perp}(\Omega) \oplus \ker(A^N + V)$, it follows that 
$\gota_V(u,v) = 0$ for all $v \in H^1(\Omega)$.
Hence $u \in \dom(A^N + V)$ and $(A^N + V) u = 0$.
So $u \in \ker(A^N + V)$.
Also $u \in H^1_{V\perp}(\Omega)$.
Therefore $u = 0$ and $T$ is injective.

The inclusion map $H^1_{V\perp}(\Omega) \subset L_2(\Omega)$ is compact 
and the form $\gotb$ is $L_2(\Omega)$-elliptic.
Hence by \cite{AEKS} Lemma~4.1 the operator $T$ is invertible.

The Sobolev embedding theorem, \cite{Nec2} Theorems~2.4.2 and 2.4.6, gives
$\Tr H^1(\Omega) \subset L_{ \frac{2d-2}{d-2+\varepsilon} }(\Gamma)$
for all $\varepsilon \in (0,1]$.
Moreover,  $L_{ \frac{2d-2}{d-2+\varepsilon} }(\Gamma) \subset L_{\frac{d}{d-1}}(\Gamma)$.
Hence there exists a $c > 0$ such that 
\[
\Big| \int_\Gamma \tau \, \overline{\tr v} \Big|
\leq c \, \|\tau\|_{L_d(\Gamma)} \, \|v\|_{H^1(\Omega)}
\]
for all $\tau \in L_d(\Gamma)$ and $v \in H^1(\Omega)$.
Now let $\tau \in L_{d,V\perp}(\Gamma)$.
Then the map
$\alpha \colon H^1_{V\perp}(\Omega) \to \Ci$ given by
$\alpha(v) = \int_\Gamma \tau \, \overline{\tr v}$ is continuous 
and anti-linear.
Hence there exists a unique $u \in H^1_{V\perp}(\Omega)$ such that 
$(T u, v)_{H^1_{V\perp}(\Omega)} = \alpha(v)$
for all $v \in H^1_{V\perp}(\Omega)$.
Moreover, $\|u\|_{H^1(\Omega)} \leq c \, \|T^{-1}\| \, \|\tau\|_{L_d(\Gamma)}$.
Then
\[
\gota_V(u,v) 
= \gotb(u,v)
= (Tu,v)_{H^1_{V\perp}(\Omega)}
= \alpha(v) 
= \int_\Gamma \tau \, \overline{\tr v}
\]
for all $v \in H^1_{V\perp}(\Omega)$.
Clearly $\gota(u,v) = 0$ and $\int_\Gamma \tau \, \overline{\tr v} = 0$ 
for all $v \in \ker(A^N + V)$.
Hence $\gota_V(u,v) = \int_\Gamma \tau \, \overline{\tr v}$ for all $v \in H^1(\Omega)$.
Note that $E \tau = u$.

`\ref{ldtnc208-3}'.
In the proof of Statement~\ref{ldtnc208-2} we deduced that 
$\|E \tau\|_{H^1(\Omega)} \leq c \, \|T^{-1}\| \, \|\tau\|_{L_d(\Gamma)}$
for all $\tau \in L_{d,V\perp}(\Gamma)$. 
So $E$ is continuous.

`\ref{ldtnc208-4}'.
This follows from \cite{Nit4} Theorem~3.14(ii).

`\ref{ldtnc208-5}'.
By \cite{Nit4} Theorem~3.14(ii) there exists a $c' > 0$ such that 
\[
\|E \tau\|_{C(\overline \Omega)}
\leq c'( \|E \tau\|_{L_2(\Omega)} + \|\tau\|_{L_d(\Gamma)})
\]
for all $\tau \in L_{d,V\perp}(\Gamma)$. 
But 
\[
\|E \tau\|_{L_2(\Omega)} 
\leq \|E \tau\|_{H^1(\Omega)} 
\leq c \, \|T^{-1}\| \, \|\tau\|_{L_d(\Gamma)}
.  \]
So $\|E \tau\|_{C(\overline \Omega)} \leq c'(c \, \|T^{-1}\| + 1) \|\tau\|_{L_d(\Gamma)}$
for all $\tau \in L_{d,V\perp}(\Gamma)$. 

`\ref{ldtnc208-6}'.
This follows from the definitions of $E$ and $D_V$.
\end{proof}

Now we are able to prove that the operator $D_{V,c}$ is densely defined.

\begin{proof}[{\bf Proof of Theorem~\ref{tdtnc201}.}]
Let $H^1_{V\perp}(\Omega)$, $L_{d,V\perp}(\Gamma)$ and the map 
$E \colon L_{d,V\perp}(\Gamma) \to H^1_{V\perp}(\Omega) \cap C(\overline \Omega)$
be as in Lemma~\ref{ldtnc208}.
Let $M > 0$ be such that 
\[
\|E \tau\|_{C(\overline \Omega)} \leq M \, \|\tau\|_{L_d(\Gamma)}
\]
for all $\tau \in L_{d,V\perp}(\Gamma)$.
Let $N \in \Ni_0$ and $u_1,\ldots,u_N \in \ker(A^N + V)$ be such that 
$u_1,\ldots,u_N$ is a basis for $\ker(A^N + V)$ and is orthonormal 
in $H^1(\Omega)$.
Note that $u_k \in C(\overline \Omega)$ for all $k \in \{ 1,\ldots,N \} $
by Lemma~\ref{ldtnc208}\ref{ldtnc208-1}.

Let $\varphi \in C(\Gamma)$ and $\varepsilon > 0$.
By Lemma~\ref{ldtnc206} there exists a $u \in C^1(\overline \Omega) \cap H^2(\Omega)$
such that $(\ca + V) u = 0$ and 
$\|u|_\Gamma - \varphi\|_{C(\Gamma)} < \varepsilon$.
Then $u$ has a weak conormal derivative and 
$\partial_\nu u \in L_\infty(\Gamma)$ by Lemma~\ref{ldtnc203}.
If $v \in \ker(A^N + V)$, then
\[
\int_\Gamma (\partial_\nu u) \, \overline{\tr v}
= \gota(u,v) - \int_\Omega (\ca u) \, \overline v
= \gota(u,v) + \int_\Omega V \, u \, \overline v
= \gota_V(u,v)
= (u, (A^N + V)v)_{L_2(\Omega)}
= 0
.  \]
So $\partial_\nu u \in L_{d,V\perp}(\Gamma)$.
By Lemma~\ref{ldtnc208}\ref{ldtnc208-1.5}
there exists a $\tau \in C(\Gamma) \cap L_{d,V\perp}(\Gamma)$ such that 
$\|\tau - \partial_\nu u\|_{L_d(\Gamma)} < \varepsilon$.
Choose $w = E \tau$.
Then $w \in H^1_{V\perp}(\Omega) \cap C(\overline \Omega)$ and 
\[
\gota_V(w,v) 
= \int_\Gamma \tau \, \overline{\Tr v}
\]
for all $v \in H^1(\Omega)$.
Set $c_k = (u,u_k)_{H^1(\Omega)} \in \Ci$ for all $k \in \{ 1,\ldots,N \} $.
Then by construction $w - u + \sum_{k=1}^N c_k \, u_k \in H^1_{V\perp}(\Omega)$.
Let $v \in H^1(\Omega)$.
Then
\begin{eqnarray*}
\gota_V \Big( w - u + \sum_{k=1}^N c_k \, u_k,v \Big)
& = & \gota_V(w,v) - \gota_V(u,v)  \\[0pt]
& = & \int_\Gamma \tau \, \overline{\Tr v}
   - \Big( \int_\Gamma (\partial_\nu u) \, \overline{\Tr v}
           + \int_\Omega ((\ca + V)u) \, \overline v 
     \Big)  \\
& = & \int_\Gamma (\tau -\partial_\nu u) \, \overline{\Tr v}
.
\end{eqnarray*}
Note that $\tau - \partial_\nu u \in L_{d,V\perp}(\Gamma)$.
So 
\[
w - u + \sum_{k=1}^N c_k \, u_k
= E (\tau - \partial_\nu u)
.  \]
Hence 
\[
\|w - u + \sum_{k=1}^N c_k \, u_k\|_{C(\overline \Omega)}
\leq M \, \|\tau -\partial_\nu u\|_{L_d(\Gamma)}
\leq M \, \varepsilon
.  \]
Then $\|w|_\Gamma - \varphi + \sum_{k=1}^N c_k \, u_k|_\Gamma\|_{C(\Gamma)} \leq (M+1) \varepsilon$.

Finally note that 
$\tr w \in \dom(D_V)$ and $D_V (\tr w) = \tau$ by Lemma~\ref{ldtnc208}\ref{ldtnc208-6}.
Since both $\tr w$ and $\tau$ are continuous, one deduces that $\Tr w \in \dom(D_{V,c})$.
Moreover, $\tr u_k \in \dom(D_{V,c})$ for all $k \in \{ 1,\ldots,N \} $ by 
Lemma~\ref{ldtnc208}\ref{ldtnc208-7}.
So $\varphi \in \overline{\dom(D_{V,c})}$.
The proof of Theorem~\ref{tdtnc201} is complete.
\end{proof}

\section{$C_0$-semigroup on $C(\Gamma)$} \label{Sdtnc4}

We next consider the problem whether $-D_{V,c}$ generates a $C_0$-semigroup 
on $C(\Gamma)$. 
If $(X,\cb,\mu)$ is a measure space, then
for operators on the Hilbert space $L_2(X)$ the notation of positivity has 
two different meanings and in the next lemma we need both of them.
We will use the following terminology if confusion is possible.
If $B$ is an operator in a Hilbert space $H$, then we say that 
$B$ is {\bf positive in the Hilbert space sense} if $(B u,u)_H \geq 0$
for all $u \in \dom(B)$.
If $B \colon L_2(X) \to L_2(X)$ is a linear operator, then we say 
that $B$ is {\bf positive in the Banach lattice sense}
if $B f \geq 0$ for all $f \in L_2(X)$ with $f \geq 0$.
Here $f \geq 0$ means that $f(x) \geq 0$ for almost all $x \in X$.
Below we consider the two cases $X = \Omega$, provided with the Lebesgue
measure, and $X = \Gamma$, provided with the $(d-1)$-dimensional Hausdorff measure.

The following proposition is known if $a_{kl} = \delta_{kl}$, that is
if $A = - \Delta$.

\begin{prop} \label{pdtnc213}
For all $k,l \in \{ 1,\ldots,d \} $ let $a_{kl} \in L_\infty(\Omega,\Ri)$.
Let $V \in L_\infty(\Omega,\Ri)$.
Suppose {\rm (\ref{eSdtnc2;3})}, {\rm (\ref{eSdtnc2;1})} and {\rm (\ref{eSdtnc2;2})}
are valid.
\begin{tabel} 
\item \label{pdtnc213-1}
Suppose that $A^D + V$ is positive in the Hilbert space sense and $0 \not\in \sigma(A^D + V)$.
Then the semigroup $S^V$ is positive in the Banach lattice sense.
\item \label{pdtnc213-2}
Suppose that $V \geq 0$.
Then the semigroup $S^V$ is submarkovian.
\end{tabel}
\end{prop}
\begin{proof}
Statement~\ref{pdtnc213-1} can be proved as in \cite{ArM} Theorem~5.1 
or \cite{EO4} Theorem 2.3(a), with obvious modifications.
Statement~\ref{pdtnc213-2} is similar to \cite{EO4} Theorem 2.3(b).
\end{proof}

It turns out that the resolvent of $D_{V,c}$ behaves well.
Recall that $D_V$ is a lower-bounded self-adjoint operator.

\begin{lemma} \label{ldtnc214}
For all $k,l \in \{ 1,\ldots,d \} $ let $a_{kl} \in L_\infty(\Omega,\Ri)$.
Let $V \in L_\infty(\Omega,\Ri)$.
Suppose {\rm (\ref{eSdtnc2;3})}, {\rm (\ref{eSdtnc2;1})} and {\rm (\ref{eSdtnc2;2})}
are valid.
Let $\omega \in \Ri$ be such that $\|S^V_t\|_{2 \to 2} \leq e^{\omega t}$ for all 
$t > 0$.
Let $\lambda \in (\omega,\infty)$.
Then one has the following.
\begin{tabel}
\item \label{ldtnc214-1}
$\lambda \, I + D_{V,c}$ is invertible.
\item \label{ldtnc214-2}
$(\lambda \, I + D_{V,c})^{-1} \psi = (\lambda \, I + D_V)^{-1} \psi$
for all $\psi \in C(\Gamma)$.
\item \label{ldtnc214-3}
If $A^D + V$ is positive in the Hilbert space sense, 
then $(\lambda \, I + D_{V,c})^{-1}$ is positive in the Banach lattice sense.
\end{tabel}
\end{lemma}
\begin{proof}
`\ref{ldtnc214-1}'.
Let $\psi \in C(\Gamma)$.
Write $\varphi = (\lambda \, I + D_V)^{-1} \psi \in L_2(\Gamma)$.
Then $D_V \varphi = \psi - \lambda \, \varphi$ and 
$\varphi \in C(\Gamma)$ by Lemma~\ref{ldtnc231}.
So $\psi - \lambda \, \varphi \in C(\Gamma)$ and $\varphi \in \dom(D_{V,c})$.
Obviously $(\lambda \, I + D_{V,c}) (\lambda \, I + D_V)^{-1} \psi = \psi$.
So the operator $\lambda \, I + D_{V,c}$ is surjective.
Since $\lambda \, I + D_V$ is injective, also the operator 
$\lambda \, I + D_{V,c}$ is injective.
Therefore $\lambda \, I + D_{V,c}$ is bijective, that is invertible.

Statement~\ref{ldtnc214-2} is now clear.

`\ref{ldtnc214-3}'.
It follows from Proposition~\ref{pdtnc213}\ref{pdtnc213-1} that the 
operator $(\lambda \, I + D_V)^{-1}$ is positive in the Banach lattice sense on $L_2(\Gamma)$.
Then the statement is a consequence of Statement~\ref{ldtnc214-2}.
\end{proof}

We now prove the main theorem of this paper.
In view of our general assumption~(\ref{eSdtnc2;2}), 
Condition~\ref{tdtnc215-3} can be reformulated by saying that the 
first eigenvalue of $A^D + V$ is strictly positive.
In contrast to this, Condition~\ref{tdtnc215-2} does not include
any spectral condition (except that $0 \not\in \sigma(A^D + V)$).
As a matter of fact, in fact the potential can be very negative.
Condition~\ref{tdtnc215-1} is a special case of Condition~\ref{tdtnc215-3}.
We give, however, diferent proofs for these two cases.
Whereas under Condition~\ref{tdtnc215-1} and \ref{tdtnc215-2}
an $L_\infty$-bound is known for the semigroup $S^V$, we
use for \ref{tdtnc215-3} that any densely defined resolvent positive 
operator on $C(\Gamma)$ is the generator of a $C_0$-semigroup.

\begin{thm} \label{tdtnc215}
For all $k,l \in \{ 1,\ldots,d \} $ let $a_{kl} \in W^{1,\infty}(\Omega,\Ri)$.
Let $V \in L_\infty(\Omega,\Ri)$.
Suppose {\rm (\ref{eSdtnc2;3})}, {\rm (\ref{eSdtnc2;1})} and {\rm (\ref{eSdtnc2;2})}
are valid.
Moreover, suppose that at least one of the following conditions is valid.
\begin{tabel}
\item \label{tdtnc215-1}
$V \geq 0$.
\item \label{tdtnc215-2}
One has $a_{kl} = \delta_{kl}$ for all $k,l \in \{ 1,\ldots,d \} $ 
and the set $\Omega$ has a $C^{1,1}$-boundary.
\item \label{tdtnc215-3}
$A^D + V$ is positive in the Hilbert space sense.
\end{tabel}
Then $S^V_t C(\Gamma) \subset C(\Gamma)$ for all $t > 0$ and 
$(S^V_t |_{C(\Gamma)})_{t > 0}$ is a $C_0$-semigroup whose generator is $-D_{V,c}$.
\end{thm}
\begin{proof} 
{\bf Case I. Suppose \ref{tdtnc215-1} or \ref{tdtnc215-2} is valid. }  
Then by 
Proposition~\ref{pdtnc213}\ref{pdtnc213-2} (in case \ref{tdtnc215-1}) 
and \cite{AE7} Proposition~6.10 (in case \ref{tdtnc215-2})
the semigroup $S^V$ leaves $L_\infty(\Gamma)$ invariant and 
there exists an $M \geq 1$ such that $\|S^V_t \varphi\|_\infty \leq M \, \|\varphi\|_\infty$
for all $t \in (0,1]$ and $\varphi \in L_\infty(\Gamma)$.
Then $\|T^V_t\|_{C(\Gamma) \to C(\Gamma)} \leq M$
for all $t \in (0,1]$.
If $\varphi \in \dom(D_{V,c})$, then 
\[
\|(I - T^V_t) \varphi\|_{C(\Gamma)}
\leq \int_0^t \|S^V_s \, D_V \varphi\|_\infty \, ds
\leq M \, t \, \|D_V \varphi\|_\infty
\]
for all $t \in (0,1]$.
Hence $\lim_{t \downarrow 0} T^V_t \varphi = \varphi$ in $C(\Gamma)$.
Since $\dom(D_{V,c})$ is dense in $C(\Gamma)$ by Theorem~\ref{tdtnc201},
one deduces that $T^V$ is a $C_0$-semigroup on $C(\Gamma)$.
It is easy to verify that $-D_{V,c}$ is the generator.

{\bf Case II.  Suppose \ref{tdtnc215-1} or \ref{tdtnc215-3} is valid. }
Then $-D_{V,c}$ is a densely defined resolvent positive operator
by Theorem~\ref{tdtnc201} and Lemma \ref{ldtnc214}\ref{ldtnc214-3}.
Moreover, the positive cone in $C(\Gamma)$ has a non-empty interior.
Hence $-D_{V,c}$ is the generator of a $C_0$-semigroup by 
\cite{Are8} Corollary~2.3.
\end{proof}

Whereas under Condition~\ref{tdtnc215-1} or \ref{tdtnc215-3} the 
semigroup $T^V$ is positive (in the Banach lattice sense), this is 
in general not the case under Condition~\ref{tdtnc215-2}, see
\cite{Daners4}.

\begin{cor} \label{cdtnc404}
For all $k,l \in \{ 1,\ldots,d \} $ let $a_{kl} \in W^{1,\infty}(\Omega,\Ri)$.
Let $V \in L_\infty(\Omega,\Ri)$.
Suppose {\rm (\ref{eSdtnc2;3})}, {\rm (\ref{eSdtnc2;1})} and {\rm (\ref{eSdtnc2;2})}
are valid.
Suppose $A^D + V$ is positive in the Hilbert space sense.
Then for all $p \in [1,\infty)$ the semigroup $S^V$ extends to a $C_0$-semigroup
on $L_p(\Gamma)$.
\end{cor}
\begin{proof}
Let $t > 0$ and $\varphi \in L_2(\Gamma)$.
Then 
\begin{eqnarray*}
\|S^V_t \varphi\|_1
& = & \sup_{\scriptstyle \psi \in C(\Gamma) \atop  
        \scriptstyle \|\psi\|_\infty \leq 1}
     |(S^V_t \varphi, \psi)_{L_2(\Gamma)}|  \\
& = & \sup_{\scriptstyle \psi \in C(\Gamma) \atop  
        \scriptstyle \|\psi\|_\infty \leq 1}
     |(\varphi, T^V_t \psi)_{L_2(\Gamma)}|
\leq \sup_{\scriptstyle \psi \in C(\Gamma) \atop  
        \scriptstyle \|\psi\|_\infty \leq 1}
     \|\varphi\|_1 \, \|T^V_t \psi\|_\infty
\leq \|T^V_t\| \, \|\varphi\|_1
.  
\end{eqnarray*}
Hence $S^V_t$ extends to a bounded operator $S^{V(1)}_t \colon L_1 \to L_1$
and $\|S^{V(1)}_t\| \leq \|T^V_t\|$.
It is easy to verify that $S^{V(1)}$ is a semigroup on $L_1$.
Moreover, $\sup_{t \in (0,1]} \|S^{V(1)}_t\| \leq \sup_{t \in (0,1]} \|T^V_t\| < \infty$.
Since $\Gamma$ has finite measure, the semigroup $S^{V(1)}$ is a
$C_0$-semigroup.
Then by duality and interpolation the corollary follows.
\end{proof}

\section{The Robin semigroup on $C(\overline \Omega)$} \label{Sdtnc6}

In order to prove irreducibility of $T^V$ in case $A^D + V$ is positive
in the Hilbert space sense, we make a detour and prove irreducibility 
for the Robin Laplacian.

Throughout this section we assume that $\Omega \subset \Ri^d$ is a bounded open 
connected set 
with Lipschitz boundary, $a_{kl} = a_{lk} \in L_\infty(\Omega,\Ri)$, 
the ellipticity condition (\ref{eSdtnc2;1}) is valid and 
$V \in L_\infty(\Omega,\Ri)$.
Moreover, let $\beta \in L_\infty(\Gamma,\Ri)$.
We do not assume that $0 \not\in \sigma(A^D + V)$.
Define the sesquilinear form 
$\gota_{V,\beta} \colon H^1(\Omega) \times H^1(\Omega) \to \Ci$ by 
\[
\gota_{V,\beta}(u,v)
= \gota_V(u,v) + \int_\Gamma \beta \, \Tr u \, \overline{\Tr v}
.  \]
Then $\gota_{V,\beta}$ is an $L_2(\Omega)$-elliptic sesquilinear form.
Let $A_{V,\beta}$ be the associated operator.
Then $A_{V,\beta}$ is self-adjoint and bounded below.
It is easy to see that 
\[
\dom(A_{V,\beta})
= \{ u \in H^1(\Omega) : \ca u \in L_2(\Omega) \mbox{ and }
     \partial_\nu u + \beta \, \Tr u = 0 \}
\]
and $A_{V,\beta} u = \ca u + V \, u$ for all $u \in \dom(A_{V,\beta})$.
So $A_{V,\beta}$ is the realisation of $\ca + V$ with Robin boundary 
conditions.
The operator $-A_{V,\beta}$ generates a $C_0$-semigroup $S^{V,\beta}$ 
on $L_2(\Omega)$, which is called the Robin semigroup.
If $\beta \geq 0$ then it is well known that the semigroup $S^{V,\beta}$ has 
Gaussian kernel bounds (see \cite{AE1} Theorem~4.9) and therefore 
the semigroup $S^{V,\beta}$ on $L_2(\Omega)$ extrapolates to a 
$C_0$-semigroup on $L_p(\Omega)$ for all $p \in [1,\infty)$.
It is an open problem whether the same is valid without the condition
$\beta \geq 0$.

The main theorem of this section is as follows.

\begin{thm} \label{tdtnc601}
Adopt the above notation and assumptions.
\begin{tabel}
\item \label{tdtnc601-1}
The semigroup $S^{V,\beta}$ is positive (in the Banach lattice sense).
\item \label{tdtnc601-1.2}
If $t > 0$ then $S^{V,\beta}_t L_2(\Omega) \subset C(\overline \Omega)$.
\item \label{tdtnc601-1.4}
If $\lambda > \omega$, then 
$(\lambda \, I + A_{V,\beta})^{-d} L_2(\Omega) \subset C(\overline \Omega)$.
Here $\omega \in \Ri$ is chosen large enough such that 
$\sup_{t > 0} e^{-\omega t} \|S^{V,\beta}_t\|_{2 \to 2} < \infty$.
\item \label{tdtnc601-2}
For all $t > 0$ the operator $S^{V,\beta}_t$ has a continuous 
kernel $k_t \colon \overline \Omega \times \overline \Omega \to \Ri$.
\item \label{tdtnc601-3}
The operator $A_{V,\beta}$ has compact resolvent.
\item \label{tdtnc601-4}
The semigroup $S^{V,\beta}$ is irreducible {\rm (}on $L_2(\Omega)${\rm )}.
\item \label{tdtnc601-4.2}
The eigenvalue $\min \sigma(A_{V,\beta})$ is simple.
\item \label{tdtnc601-4.5}
The semigroup $(S^{V,\beta}_t|_{C(\overline \Omega)})_{t > 0}$
is a $C_0$-semigroup on $C(\overline \Omega)$.
\item \label{tdtnc601-5}
The semigroup $(S^{V,\beta}_t|_{C(\overline \Omega)})_{t > 0}$ is irreducible 
{\rm (}on $C(\overline \Omega)${\rm )}.
\item \label{tdtnc601-6}
There exits a $\delta > 0$ such that $u_1(x) \geq \delta$
for all $x \in \overline \Omega$, where 
$u_1 \in L_2(\overline \Omega)$ is an eigenfunction
of $A_{V,\beta}$ with eigenvalue $\min \sigma(A_{V,\beta})$
such that $u_1 \geq 0$ almost everywhere.
\item \label{tdtnc601-7}
For all $p \in [1,\infty)$ the semigroup $S^{V,\beta}$
extends consistently to a $C_0$-semigroup on $L_p(\Omega)$.
\end{tabel}
\end{thm}
\begin{proof}
`\ref{tdtnc601-1}'. 
This follows as in the proof of \cite{AE1} Theorem~4.9.
The positivity of $\beta$ is not needed in that proof.

`\ref{tdtnc601-1.2}'. 
This follows from \cite{Nit4} Theorem~3.14(ii) and 
Theorem~\ref{tdtnc240}\ref{tdtnc240-2}$\Rightarrow$\ref{tdtnc240-3}.

`\ref{tdtnc601-1.4}'. 
This follows from \cite{Nit4} Lemmas 3.11 and 3.10.

`\ref{tdtnc601-2}'. 
This is a consequence of Statement~\ref{tdtnc601-1.2} and Theorem~\ref{tdtnc240}.

`\ref{tdtnc601-3}'. 
Easy.

`\ref{tdtnc601-4}'. 
This is a consequence of \cite{Ouh5} Corollary~2.11.

`\ref{tdtnc601-4.2}'.
See Lemma~\ref{ldtnc528}\ref{ldtnc528-4}.
 
`\ref{tdtnc601-4.5}'. 
Define the part $A_{V,\beta,c}$ of $A_{V,\beta}$ in $C(\overline \Omega)$ by
\[
\dom(A_{V,\beta,c})
= \{ u \in C(\overline \Omega) \cap \dom(A_{V,\beta}) : 
         A_{V,\beta} u \in C(\overline \Omega) \}
\]
and $A_{V,\beta,c} u = A_{V,\beta} u$ for all $u \in \dom(A_{V,\beta,c})$.
Then $\dom(A_{V,\beta,c})$ is dense in $C(\overline \Omega)$ by 
the arguments in the proof of Theorem~4.3 in \cite{Nit4}.
(Remark, unfortunately there is a gap in the proof of Theorem~4.3 in \cite{Nit4}
for the part that the restriction $(S^{V,\beta}_t|_{C(\overline \Omega)})_{t > 0}$
of the Robin semigroup in $C(\overline \Omega)$ is a $C_0$-semigroup,
since it is unclear whether 
$\sup_{t \in (0,1]} \|S^{V,\beta}_t\|_{\infty \to \infty} < \infty$.
He used that the semigroup $S^{V,\beta}$ has a kernel with Gaussian bounds, which 
is only known in case $\beta \geq 0$.)

Let $\omega \in \Ri$ be as in Statement~\ref{tdtnc601-1.4}.
Let $\lambda > \omega$.
Then the operator $\lambda \,I + A_{V,\beta,c}$
is invertible by the same argument as in the proof of 
Lemma~\ref{ldtnc214}\ref{ldtnc214-1}.
Since the resolvent operator $(\lambda \,I + A_{V,\beta})^{-1}$ is 
positive on $L_2(\Omega)$, also the resolvent operator 
$(\lambda \,I + A_{V,\beta,c})^{-1}$ is 
positive on $C(\overline \Omega)$.
Moreover, the positive cone in $C(\overline \Omega)$ has a non-empty interior.
Hence $-A_{V,\beta,c}$ is the generator of a $C_0$-semigroup by 
\cite{Are8} Corollary~2.3.

`\ref{tdtnc601-5}' and `\ref{tdtnc601-6}'. 
This follows from Corollary~\ref{cdtnc522.5}.

`\ref{tdtnc601-7}'.
The proof is similarly to the proof of Corollary~\ref{cdtnc404}.
\end{proof}

\begin{remark} \label{rdtnc602}
In order to avoid confusion with the assumptions and notation in 
the rest of this paper we continued to assume in this section that the coefficients
are symmetric and that there are no first-order terms.
One can, however, consider the full Robin form 
$\gota \colon H^1(\Omega) \times H^1(\Omega) \to \Ci$ given by
\[
\gota(u,v) 
= \sum_{k,l=1}^d \int_\Omega a_{kl} \, (\partial_k u) \, \overline{\partial_l v}
   + \sum_{k=1}^d \int_\Omega b_k \, (\partial_k u) \, \overline v
   + \sum_{k=1}^d \int_\Omega c_k \, u \, \overline{\partial_k v}
   + \int_\Omega c_0 \, u \, \overline v
   + \int_\Gamma \beta \, \Tr u \, \overline{\Tr v}
,  \]
where $a_{kl}, b_k, c_k, c_0 \in L_\infty(\Omega,\Ri)$ and 
$\beta \in L_\infty(\Gamma,\Ri)$, together with the 
ellipticity condition~(\ref{eSdtnc2;1}).
We do not assume any longer that the $a_{kl}$ are symmetric.
Let $A$ be the m-sectorial operator associated with $\gota$ 
and let $S$ be the semigroup generated by $-A$ on $L_2(\Omega)$.
Then Statements~\ref{tdtnc601-1}, \ref{tdtnc601-1.2}, 
\ref{tdtnc601-1.4}, \ref{tdtnc601-2},
\ref{tdtnc601-3}, \ref{tdtnc601-4}, 
\ref{tdtnc601-4.5} and \ref{tdtnc601-7} are still valid, with the 
same proof.
Instead of Statement~\ref{tdtnc601-4.2} one can consider
$\lambda_1 = \inf \{ \RRe \lambda : \lambda \in \sigma(A) \} $.
Then $\lambda_1 \in \sigma(A)$ by \cite{ABHN} Proposition~3.11.2 and 
it follows as before that $\lambda_1$ is a simple eigenvalue.
If $A$ is symmetric, then also Statement~\ref{tdtnc601-6}
is valid.

We do not know whether Statement~\ref{tdtnc601-5} is
still valid if $A$ is not symmetric.
We also do not know whether Statement~\ref{tdtnc601-7} is valid if 
$b_k, c_k, c_0 \in L_\infty(\Omega)$ and 
$\beta \in L_\infty(\Gamma)$ are complex valued.
\end{remark}

\section{Strictly positive first eigenfunction and extensions to $L_p(\Gamma)$} \label{Sdtnc7}

In this section we consider the case where the semigroup generated by $-D_V$
is positive (in the Banach lattice sense) and under the condition 
that $\Omega$ is connected we show that the first eigenfunction
is strictly positive.
We deduce from this that the Dirichlet-to-Neumann semigroup is 
irreducible on $C(\Gamma)$.
This is surprising since we merely assume that $\Omega$ is connected. 
For example, if $\Omega$ is an annulus, then $\Gamma$ is not connected.
The result also allows us to extend the semigroup $S^V$ consistently 
to a $C_0$-semigroup on $L_p(\Gamma)$ for all $p \in [1,\infty)$.

We adopt the assumptions and notation as in Section~\ref{Sdtnc2}.
In particular, 
for all $k,l \in \{ 1,\ldots,d \} $ let $a_{kl} \in L_\infty(\Omega,\Ri)$.
Let $V \in L_\infty(\Omega,\Ri)$.
We suppose that {\rm (\ref{eSdtnc2;3})}, {\rm (\ref{eSdtnc2;1})} and {\rm (\ref{eSdtnc2;2})}
are valid.
In addition we assume that $\Omega$ is connected and that 
$A^D + V$ is positive (in the Hilbert space sense).
Then $S^V$ is a positive semigroup by Proposition~\ref{pdtnc213}\ref{pdtnc213-1}.
Moreover, $S^V_t L_2(\Gamma) \subset C(\Gamma)$ for all $t > 0$
by Proposition~\ref{pdtnc232} and $D_V$ is self-adjoint with 
compact resolvent.
So all eigenfunctions of $D_V$ are elements of $C(\Gamma)$.
Let $\lambda_1 = \min \sigma(D_V)$.
Let $\varphi_1 \in C(\Gamma)$ be an eigenfunction with eigenvalue $\lambda_1$
such that $\varphi_1 \geq 0$.

\begin{thm} \label{tdtnc701}
Adopt the above notation and assumptions.
Then $\min \varphi_1 > 0$.
\end{thm}

The theorem is an immediate consequence of the next proposition.

\begin{prop} \label{pdtnc702}
Let $\beta \in \Ri$ and $\varphi \in \dom(D_V)$ be 
and eigenfunction of $D_V$ with eigenvalue $-\beta$.
Suppose that $\varphi \geq 0$. 
We identify the real number $\beta$ with the constant function $\beta \, \one_\Gamma$
on $\Gamma$.
Consider the Robin operator $A_{V,\beta}$ as in Section~\ref{Sdtnc6}.
Then $\min \sigma(A_{V,\beta}) = 0$.
Let $u_1 \in C(\overline \Omega)$ be an eigenfunction
of $A_{V,\beta}$ with eigenvalue $0$
as in Theorem~\ref{tdtnc601}\ref{tdtnc601-6}.
Then there exists a $c > 0$ such that 
$\varphi = c \, u_1|_\Gamma$.
In particular, 
$\dim \spann \{ \psi \in \ker(\beta \, I + D_V) : \psi \geq 0 \} = 1$
and $\varphi(z) > 0$ for all $z \in \Gamma$.
\end{prop}
\begin{proof}
By definition of $D_V$ there exists a $u \in H^1(\Omega)$ such that 
$\Tr u = \varphi$ and 
\begin{equation}
\gota_V(u,v) = - \beta \, (\varphi, \Tr v)_{L_2(\Gamma)}
\label{epdtnc702;1}
\end{equation}
for all $v \in H^1(\Omega)$.
Since $\varphi \geq 0$ it follows that $u$ is real valued and 
$u^- \in H^1_0(\Omega)$.
Choose $v = u^-$.
Then $\gota_V(u,u^-) = 0$.
But $\partial_k (u^-) = - (\partial_k u) \, \one_{[u < 0]}$
for all $k \in \{ 1,\ldots,d \} $.
Therefore $\gota_V(u^-,u^-) = \gota_V(u,u^-) = 0$.
Since $A^D + V$ is a positive operator in the Hilbert space sense 
with trivial kernel by assumption (\ref{eSdtnc2;2}),
it follows that $u^- = 0$.
Therefore $u \geq 0$ and clearly $u \neq 0$.
It follows from (\ref{epdtnc702;1}) that $\gota_{V,\beta}(u,v) = 0$
for all $v \in H^1(\Omega)$.
Therefore $u \in \dom(A_{V,\beta})$ and $A_{V,\beta} u = 0$.
The operator $-A_{V,\beta}$ is self-adjoint, has compact resolvent 
and generates a positive irreducible semigroup in $L_2(\Omega)$.
Hence it follows from the inverse Krein--Rutman theorem
\cite{AE6} Lemma~5.14 that $0 = \min \sigma(A_{V,\beta})$.

Since $\min \sigma(A_{V,\beta})$ is a simple eigenvalue of $A_{V,\beta}$
by Theorem~\ref{tdtnc601}\ref{tdtnc601-4.2}, 
it follows that 
there exists a $c \in \Ci \setminus \{ 0 \} $ such that $u = c \, u_1$.
But both $u_1,u \geq 0$.
Therefore $c > 0$.
Then $\varphi = \Tr u = c \, \Tr u_1$.
Since $u_1(x) > 0$ for all $x \in \overline \Omega$ by 
Theorem~\ref{tdtnc601}\ref{tdtnc601-6}, obviously $\varphi(z) > 0$
for all $z \in \Gamma$.
\end{proof}

Recall that $T^V$ is the restriction of the Dirichlet-to-Neumann semigroup $S^V$
to $C(\Gamma)$.
We show below that $T^V$ is irreducible.
We cannot deduce this in general from the strict positivity of 
the first eigenfunction via Proposition~\ref{pdtnc522},
since $\Gamma$ is not connected in general.

Irreducibility of $S^V$ in $L_2(\Gamma)$ is much easier.
We need the following result.

\begin{prop} \label{pdtnc703}
Let $(Y,\Sigma,\mu)$ be a finite measure space.
Let $B$ be a lower bounded self-adjoint operator in $L_2(Y)$
and suppose that $-B$ generates a 
positive $C_0$-semigroup $S$ on $L_2(Y)$.
Let $\varphi \in L_2(Y)$ and suppose that $\varphi(y) > 0$
for almost every $y \in Y$.
Further suppose that $S_t \varphi = \varphi$ for all $t > 0$
and that $\dim \spann \{ \psi \in \ker B : \psi \geq 0 \} = 1$.
Then $S$ is irreducible.
\end{prop}
\begin{proof}
The proof is a variation of the proof of Proposition~2.2 in \cite{AE3}.
Let $Y_1$ be a measurable subset of $Y$ and 
suppose that $S_t L_2(Y_1) \subset L_2(Y_1)$ for all $t > 0$.
Set $Y_2 = Y \setminus Y_1$.
Since $S_t$ is self-adjoint one deduces that 
$S_t L_2(Y_2) \subset L_2(Y_2)$ for all $t > 0$.
Let $t > 0$.
Then 
\[
\varphi \, \one_{Y_1} + \varphi \, \one_{Y_2}
= \varphi
= S_t \varphi
= S_t (\varphi \, \one_{Y_1}) + S_t (\varphi \, \one_{Y_2})
. \]
Since $S_t$ leaves $L_2(Y_1)$ and $L_2(Y_2)$ invariant, it follows that 
$S_t (\varphi \, \one_{Y_1}) = \varphi \, \one_{Y_1}$ and 
$S_t (\varphi \, \one_{Y_2}) = \varphi \, \one_{Y_2}$.
So $\varphi \, \one_{Y_1} \in \ker B$ and $\varphi \, \one_{Y_2} \in \ker B$.
Since $\dim \spann \{ \psi \in \ker B : \psi \geq 0 \} = 1$ one deduces that 
$\varphi \, \one_{Y_1} = 0$ or $\varphi \, \one_{Y_2} = 0$.
Therefore $\mu(Y_1) = 0$ or $\mu(Y_2) = 0$.
\end{proof}

\begin{prop} \label{pdtnc704}
The semigroup $S^V$ is irreducible on $L_2(\Gamma)$ and 
$\min \sigma(D_V)$ is a simple eigenvalue.
\end{prop}
\begin{proof}
It follows from Proposition~\ref{pdtnc702} that 
$\varphi_1(z) > 0$ for all $z \in \Gamma$ and 
$\dim \spann \{ \psi \in \ker(D_V - \lambda_1 \, I) : \psi \geq 0 \} = 1$.
Apply Proposition~\ref{pdtnc703} to the operator $D_V - \lambda_1 \, I$.
One deduces that $S^V$ is irreducible.
Then the eigenvalue $\min \sigma(D_V)$ is simple by Lemma~\ref{ldtnc528}\ref{ldtnc528-4}.
\end{proof}

Now we prove the irreducibility in $C(\Gamma)$.

\begin{thm} \label{tdtnc705}
The semigroup $T^V$ is irreducible on $C(\Gamma)$.
\end{thm}
\begin{proof}
Let $\Gamma_1$ be a closed subset of $\Gamma$ with 
$\emptyset \neq \Gamma_1 \neq \Gamma$.
Set $I = \{ \varphi \in C(\Gamma) : \varphi|_{\Gamma_1} = 0 \} $.
Assume that $T^V_t I \subset I$ for all $t > 0$.
We consider two cases.

{\bf Case I. Suppose $\Gamma_1$ is not open. }
Then there exists an $x_0 \in \partial \Gamma_1$.
Then one can argue as in the proof of the implication 
\ref{pdtnc522-2}$\Rightarrow$\ref{pdtnc522-1} in the proof 
of Proposition~\ref{pdtnc522} to deduce that 
$(T^V_t u)(x_0) = 0$ for all $u \in C(\Gamma)$ and $t > 0$.
But $(T^V_t \varphi_1)(x_0) = e^{-\lambda_1 t} \, \varphi_1(x_0) > 0$
for all $t > 0$ by Theorem~\ref{tdtnc701}.
This is a contradiction.

{\bf Case II. Suppose $\Gamma_1$ is open. }
Then $\Gamma_1$ is a connected component of $\Gamma$.
Hence $\sigma(\Gamma_1) > 0$ and $\sigma(\Gamma \setminus \Gamma_1) > 0$.
Let $J = \{ \varphi \in L_2(\Gamma) : \varphi|_{\Gamma_1} = 0 \} $.
Then $J$ is the closure of $I$ in $L_2(\Gamma)$ and 
$S^V_t J \subset J$ for all $t > 0$.
Since $S^V$ is irreducible one deduces that 
$\sigma(\Gamma_1) = 0$ or $\sigma(\Gamma \setminus \Gamma_1) = 0$.
This is a contradiction.
\end{proof}

\begin{cor} \label{cdtnc707}
For all $p \in [1,\infty)$ the semigroup $S^V$ extends 
consistently to a $C_0$-semigroup on $L_p(\Gamma)$.
\end{cor}
\begin{proof}
This follows from Proposition~\ref{pdtnc524}.
\end{proof}

\begin{cor} \label{cdtnc706}
Let $\varphi \in C(\Gamma)$ with $\varphi \geq 0$ and 
$\varphi \neq 0$.
Then $(T^V_t \varphi)(z) > 0$ for all $t > 0$ and $z \in \Gamma$.
\end{cor}
\begin{proof}
Apply Proposition~\ref{pdtnc522}\ref{pdtnc522-1}$\Rightarrow$\ref{pdtnc522-4}.
\end{proof}

We do not know whether $T^V$ is a $C_0$-semigroup 
(unless the $a_{kl}$ are Lipschitz continuous, see Theorem~\ref{tdtnc215}\ref{tdtnc215-3}).

\subsection*{Acknowledgements}
The first-named author is most grateful for the hospitality extended
to him during a fruitful stay at the University of Auckland and the 
second-named author
for a wonderful stay at the University of Ulm.
Part of this work is supported by an
NZ-EU IRSES counterpart fund and the Marsden Fund Council from Government funding,
administered by the Royal Society of New Zealand.
Part of this work is supported by the
EU Marie Curie IRSES program, project `AOS', No.~318910.

\end{document}